\pdfoutput=1
\documentclass[11pt]{article}
\usepackage[left=1in,right=1in,top=1in,bottom=1in]{geometry}
\usepackage{times}
\usepackage{expl3}
\usepackage{cite}
\usepackage[table]{xcolor}
\usepackage{multirow}
\usepackage{stackengine} 
\usepackage{hhline}
\usepackage{lipsum}
\usepackage{titlesec}
\usepackage{wrapfig}
\usepackage{enumerate}
\usepackage{epsfig}
\usepackage{tikz-cd}
\usepackage{amsmath}
\usepackage{tabularx}
\usepackage{array}
\usepackage{booktabs}
\usepackage{enumitem}
\usepackage{bbm}
\usepackage{calc}
\usepackage{graphicx}
\usepackage{amsmath}
\usepackage[title]{appendix}
\usepackage{amssymb}
\usepackage{epstopdf}
\usepackage{boldline}
\usepackage{arydshln}
\usepackage{calligra}
\usepackage{bm}
\usepackage{url}
\usepackage{blindtext}
\usepackage{accents}
\usepackage{amsthm}

\newcommand{\HilbMult}{\mathrm{HilbMult}}

\newtheorem{definition}{Definition}
\newtheorem{assumption}{Assumption}
\newtheorem{axiom}{Axiom}
\newtheorem{theorem}{Theorem}

\newtheorem{corollary}{Corollary}
\newtheorem{lemma}{Lemma}
\newtheorem{example}{Example}
\newtheorem{remark}{Remark}
\usepackage{mathtools}
\usepackage{epstopdf}
\usepackage{balance}
\usepackage{thmtools}
\usepackage{thm-restate}
\usepackage{hyperref}
\usepackage{cleveref}
\usepackage[mathscr]{euscript}

\usepackage[ruled,vlined]{algorithm2e}
\newcommand{\ostar}{\mathbin{\mathpalette\make@circled\star}}

\makeatletter
\newcommand{\removelatexerror}{\let\@latex@error\@gobble}
\makeatother
\makeatletter
\newcommand*{\rom}[1]{\expandafter\@slowromancap\romannumeral #1@}
\makeatother

\ExplSyntaxOn
\newcommand\latinabbrev[1]{
  \peek_meaning:NTF . {% Same as \@ifnextchar
    #1\@}%
  { \peek_catcode:NTF a {% Check whether next char has same catcode as \'a, i.e., is a letter
      #1.\@ }%
    {#1.\@}}}
\ExplSyntaxOff

\titleclass{\subsubsubsection}{straight}[\subsubsection]

\begin{document}
\vspace{1cm}
\title{Composition and Coherence: The Syntax of Operator Networks}
% \title{Generalized Operator Integrals: Multiple Case}
\vspace{1.8cm}
\author{Shih-Yu~Chang
% <-this % stops a space
\thanks{Shih-Yu Chang is with the Department of Applied Data Science,
San Jose State University, San Jose, CA, U. S. A. (e-mail: {\tt
shihyu.chang@sjsu.edu})
}}

\maketitle

\begin{abstract}
Coherence is a central issue in category and multicategory theory, ensuring that formally distinct compositions of morphisms—such as tensor reorderings or diagrammatic rewiring—represent the same underlying transformation. In operator theory and Hilbert-space–based systems, coherence guarantees that equivalent operator networks yield identical effects on states and signals, maintaining both mathematical consistency and computational interpretability. This paper, the second installment of the Categorical Spectral Architecture (CSA) program, advances a unified framework integrating operator theory, spectral analysis, and categorical algebra. Building upon the functorial and spectral foundations established in the first paper, we introduce the Synergy Operad, a syntactic framework extending the multicategory of Hilbert spaces and bounded multilinear maps. The Canonical Representation Theorem constructs a structure-preserving functor that aligns syntactic compositions with operator semantics. Another important pillar of the work about the Coherence Theorem ensures that grammatical equivalences correspond to identical semantic results. Together these results identify a syntactic-semantic duality at the core of CSA providing a rigorous basis for analyzing and constructing complex systems of operators with coherent and predictable meaning.
\end{abstract}

\section{Introduction}\label{sec:Introduction}

% why coherence
Coherence is a foundational issue in both category theory and multicategory theory, ensuring that formally distinct compositions of morphisms—arising from permutations, tensor reorderings, or diagrammatic rewiring—represent the same underlying transformation. In the context of operator theory and Hilbert-space–based systems, coherence becomes crucial: it guarantees that equivalent wiring diagrams or operator compositions yield semantically identical actions on states and signals. Without coherence, schematic inference would be syntactically expressive but semantically UNSTABLE, undermining both mathematical consistency and the computational interpretability of complex operator networks. Creating coherence thus bridges the gap between abstract categorical syntax and analytic semantics, ensuring that compositional reasoning remains intact and predictable~\cite{awodey2010category}.

This paper represents the second part of our ongoing research program, \textbf{Categorical Spectral Architecture (CSA)}, that aims to unify operator theory, spectroscopy, and categorical algebra into a coherent mathematical framework proposed by~\cite{Chang2025HilbMultABranchE}. Building on the foundations of the functional and spectral polynomial theory , theory established in the first paper - where the functional calculus is formulated as a Banach-enriched polynomial - this work develops \emph{Synergy Operad} $\mathbf{S}$, a syntactic structure with the symmetric monoidal multicategory $\mathbf{HilbMult}$ of Hilbert spaces and bounded multilinear maps. We then construct a functor in \textbf{Canonical Representation Theorem} that preserves the structure $\rho: \mathbf{S} \to \mathbf{HilbMult}$, ensuring a strict correspondence between syntactic schemes and operator-theoretic semantics. The accompanying \textbf{Coherence Theorem} ensures the semantic integrity of syntactic rewiring, ensuring that equivalent schematic constructions produce identical operator operations. Together, these contributions complement the syntactic-semantic dualism , dualism at the heart , heart of CSA and provide a definitive foundation for operator thepry—allowing robust schematic reasoning, feedback representation, and coherence analysis even in nonlinear, infinite-dimensional, or feedback-driven settings.

The rest of the work can be presented in the following order. Section ~\ref{sec:Definitions and Axioms} presents the basic categorical definitions and axioms of $\mathbf{HilbMult}$. The definition of synergy process $\mathbf{S}$ is included as a syntactic category of operator wiring diagrams in Section~\ref{sec:The synergy operad}. Then, in Section~\ref{sec:The Canonical Representation}, we present , present a theory of  canonical representation that , that creates a functorial bridge , bridge between syntax and semantics. Then we prove the coherence theorem Section~\ref{sec:The Coherence Theorem} to describe the semantic consistency of diagrammatic equivalences.  In Section~\ref{sec:Applications in Algebra, Analysis Geometry and Topology}, we discuss  applications of $\mathbf{S}$ in algebra, analysis, geometry, and topology. Finally, we present a specific case study on cohesion in PDE operator networks in Section~\ref{sec:Case Study: Coherence in PDE Operator Networks}.

\begin{remark}
The author is solely responsible for the mathematical insights and theoretical directions proposed in this work. AI tools, including OpenAI's ChatGPT and DeepSeek models, were employed solely to assist in verifying ideas, organizing references, and ensuring internal consistency of exposition~\cite{chatgpt2025,deepseek2025}.
\end{remark}

\section{Definitions and Axioms}\label{sec:Definitions and Axioms}

Let $\mathbf{Hilb}$ denote the category of complex Hilbert spaces with bounded linear maps, and let $\HilbMult$ denote the symmetric monoidal multicategory enriched over Banach spaces. This section begins with the core definitions and axioms underlying its construction. For completeness, we restate certain foundational materials that overlap with our first paper in this series~\cite{Chang2025HilbMultABranchE}, ensuring the present exposition remains self-contained.

\subsection{Core Definitions (Briefly)}

\begin{definition}[Multicategory]
A multicategory consists of:
\begin{itemize}
    \item A collection of objects.
    \item For each finite list of objects $(X_1,\dots,X_n)$ and object $Y$, a set $\mathrm{Hom}(X_1,\dots,X_n;Y)$ of multimorphisms.
    \item Composition operations satisfying associativity.
    \item Identity morphisms for each object.
\end{itemize}
\end{definition}

\begin{definition}[Banach Enrichment]
A category/multicategory is enriched in Banach spaces if each hom-set $\mathrm{Hom}(X_1,\dots,X_n;Y)$ is a Banach space and composition is a bounded multilinear map.
\end{definition}

\begin{definition}[Symmetric Monoidal Multicategory]
A symmetric monoidal multicategory has:
\begin{itemize}
    \item A tensor product $\otimes$ on objects.
    \item Natural isomorphisms for associativity, unit, and symmetry.
    \item Compatibility between tensor product and multimorphisms.
\end{itemize}
\end{definition}

% ====== CORE AXIOMS ======
\subsection{Core Axioms}

\begin{axiom}[Objects and Hom-Spaces]\label{ax:H1-H2}
\begin{enumerate}[label=(\alph*)]
    \item \textbf{Objects:} The objects of $\mathrm{HilbMult}$ are (separable) complex Hilbert spaces $H, K, \dots$.
    
    \item \textbf{Hom-Spaces:} For each finite tuple $(H_1,\dots,H_n;K)$, there is a Banach space:
    \[
    \mathrm{Hom}(H_1,\dots,H_n;K)
    \]
    whose elements are (equivalence classes of) bounded multilinear maps $T: H_1\times\cdots\times H_n \to K$. In the unary case ($n=1$), these are bounded linear maps $H \to K$.
    
    \item \textbf{Norm:} The Banach norm is the operator norm:
    \[
    \|T\| := \sup_{\|x_i\|\le 1} \|T(x_1,\dots,x_n)\|
    \]
\end{enumerate}
\end{axiom}

\begin{axiom}[Composition and Identities]\label{ax:H3-H4}
\begin{enumerate}[label=(\alph*)]
    \item \textbf{Multilinear Composition:} For composable multimorphisms:
    \begin{align*}
        &S \in \mathrm{Hom}(K_1,\dots,K_m;L), \\
        &T_j \in \mathrm{Hom}(H_{j,1},\dots,H_{j,n_j};K_j) \quad (1\le j\le m),
    \end{align*}
    there is a composition $S\circ (T_1,\dots,T_m) \in \mathrm{Hom}(H_{1,1},\dots,H_{m,n_m};L)$ given by:
    \[
    (S\circ (T_j))(\mathbf{x}) := S(T_1(\mathbf{x}_1),\dots,T_m(\mathbf{x}_m))
    \]
    Composition is multilinear and contractive:
    \[
    \|S\circ(T_1,\dots,T_m)\| \le \|S\|\cdot\prod_{j=1}^m \|T_j\|
    \]
    
    \item \textbf{Identities:} For each object $H$, there is an identity $\mathrm{id}_H \in \mathrm{Hom}(H;H)$ given by $\mathrm{id}_H(x) = x$. These satisfy:
    \[
    T \circ (\mathrm{id}_{H_1}, \dots, \mathrm{id}_{H_n}) = T
    \]
    for any $T \in \mathrm{Hom}(H_1,\dots,H_n;K)$.
\end{enumerate}
\end{axiom}

% ====== MONOIDAL STRUCTURE ======
\paragraph{Monoidal Structure}

\begin{axiom}[Symmetric Monoidal Structure]\label{ax:H5}
$\mathrm{HilbMult}$ has a symmetric monoidal structure:
\begin{enumerate}[label=(\alph*)]
    \item \textbf{Tensor Product:} 
    \begin{itemize}
        \item On objects: $H \otimes K$ is the completed Hilbert space tensor product.
        \item Unit object: $\mathbb{C}$ (complex numbers as 1D Hilbert space).
        \item On multimorphisms: For $T \in \mathrm{Hom}(H_1,\dots,H_n;K)$ and $T' \in \mathrm{Hom}(H'_1,\dots,H'_p;K')$, their tensor product is:
        \[
        T \otimes T' \in \mathrm{Hom}(H_1,\dots,H_n,H'_1,\dots,H'_p; K \otimes K')
        \]
        defined by $(T \otimes T')(\mathbf{x}, \mathbf{x}') = T(\mathbf{x}) \otimes T'(\mathbf{x}')$.
    \end{itemize}
    
    \item \textbf{Structural Isomorphisms:} Natural unitary isomorphisms:
    \begin{align*}
        \text{Associator: } &\alpha_{H,K,L}: (H \otimes K) \otimes L \xrightarrow{\cong} H \otimes (K \otimes L) \\
        \text{Braiding: } &\sigma_{H,K}: H \otimes K \xrightarrow{\cong} K \otimes H \\
        \text{Unitors: } &\lambda_H: \mathbb{C} \otimes H \xrightarrow{\cong} H, \quad \rho_H: H \otimes \mathbb{C} \xrightarrow{\cong} H
    \end{align*}
    satisfying the coherence conditions (pentagon, hexagon, triangle identities).
    
    \item \textbf{Permutation Action:} For $T \in \mathrm{Hom}(H_1,\dots,H_n;K)$ and $\pi \in S_n$:
    \[
    T^\pi(x_1,\dots,x_n) = T(x_{\pi(1)},\dots,x_{\pi(n)})
    \]
\end{enumerate}
All structural isomorphisms are isometric.
\end{axiom}

% ====== ADVANCED STRUCTURE ======
\paragraph{Advanced Structure}

\begin{axiom}[Closed Structure / Currying]\label{ax:H6}
$\mathrm{HilbMult}$ is closed under currying operations with natural isometric isomorphisms:

\begin{enumerate}[label=(\alph*)]
    \item \textbf{Basic Currying:}
    \[
    \mathrm{Hom}(H_1, \dots, H_n; K) \cong \mathrm{Hom}(H_1; \mathrm{Hom}(H_2, \dots, H_n; K))
    \]
    via $T \mapsto \Lambda_1(T)$ where $\Lambda_1(T)(x_1)(x_2, \dots, x_n) = T(x_1, x_2, \dots, x_n)$.
    
    \item \textbf{Tensor-Curry Duality:}
    \[
    \mathrm{Hom}(H \otimes K; L) \cong \mathrm{Hom}(H; \mathrm{Hom}(K; L))
    \]
    via $F \mapsto \Lambda(F)$ where $\Lambda(F)(x)(y) = F(x \otimes y)$.
    
    \item \textbf{Partial Currying:} For any $1 \le j \le n$:
    \[
    \mathrm{Hom}(H_1, \dots, H_n; K) \cong \mathrm{Hom}(H_1, \dots, H_{j-1}; \mathrm{Hom}(H_j, \dots, H_n; K))
    \]
\end{enumerate}
All isomorphisms preserve norms and are compatible with composition.
\end{axiom}

\begin{axiom}[Optional $C^*$-Structure]\label{ax:H7}
We may optionally endow endomorphism spaces with $C^*$-algebra structure:

\begin{enumerate}[label=(\alph*)]
    \item \textbf{Involution:} For each $H$, the space $\mathcal{B}(H) = \mathrm{Hom}(H;H)$ has a conjugate-linear involution $*$ satisfying:
    \begin{align*}
        (f^*)^* &= f, \quad (f \circ g)^* = g^* \circ f^*, \quad (\lambda f)^* = \bar{\lambda} f^*,
    \end{align*}
    where $\bar{\lambda}$ is the complex conjugate for the complex number $\lambda$.
    
    \item \textbf{$C^*$-Property:} The norm satisfies:
    \[
    \|f^* \circ f\| = \|f\|^2 \quad \text{and} \quad \|f \circ g\| \le \|f\|\|g\|
    \]
    
    \item \textbf{Monoidal Compatibility:} For $f \in \mathrm{Hom}(H_1; H_2)$, $g \in \mathrm{Hom}(K_1; K_2)$:
    \[
    (f \otimes g)^* = f^* \otimes g^*
    \]
\end{enumerate}

This enables the definition of:
\begin{itemize}
    \item \textbf{Self-adjoint elements:} $f^* = f$
    \item \textbf{Normal elements:} $f^* \circ f = f \circ f^*$  
    \item \textbf{Positive elements:} $f = g^* \circ g$ for some $g$
\end{itemize}
Essential for quantum theory (states, observables, measurements).
\end{axiom}

\section{The synergy operad}\label{sec:The synergy operad}

Building upon the foundational $\mathbf{HilbMult}$ multicategory, we introduce the synergy operad $\mathbf{S}$. The categorical hierarchy is:
\[
\mathbf{Hilb} \subseteq \mathbf{HilbMult} \subseteq \mathbf{S}, \quad \rho: \mathbf{S} \to \mathbf{HilbMult}.
\]
We next define the synergy operad as a syntactic category of operator wiring diagrams, whose generators and relations encode composition laws compatible with the symmetric monoidal structure of $\mathbf{HilbMult}$.

\subsection*{Axiomatic Definition of $\mathbf{S}$}

\begin{axiom}[Objects of $\mathbf{S}$]\label{ax:S1}
The objects of $\mathbf{S}$ are finite, ordered lists (tuples) of Hilbert spaces, denoted $\vec{H} = (H_1, \dots, H_n)$. Each list represents the collection of input/output ``ports" of a system. The empty list $()$ is permitted.
\end{axiom}

\subsubsection*{Well-Posed Feedback}

Let $H_i: U_i \to Y_i$ and $H_j: U_j \to Y_j$ be two systems with compatible input-output spaces, namely $Y_i = U_j$. Their feedback interconnection is governed by the equations:
\[
\begin{aligned}
y_i &= H_i(y_j), \\
y_j &= H_j(u_j - y_i).
\end{aligned}
\]
Substituting the first equation into the second yields a fixed-point equation in $y_j$:
\[
y_j = H_j(u_j - H_i(y_j)). \tag{1}
\]
Define the closed-loop mapping $\Phi_{u_j}: Y_j \to Y_j$ for a given input $u_j$ as
\[
\Phi_{u_j}(y_j) \coloneqq H_j(u_j - H_i(y_j)).
\]
The feedback interconnection is \emph{well-posed} if for every admissible input $u_j$, the equation $\Phi_{u_j}(y_j) = y_j$ admits a unique solution.

\begin{definition}[Well-Posed Feedback]\label{def:Well-Posed Feedback}
The feedback $\mathcal{F}_{i,j}$ is \textbf{well-posed} if for every $u_j \in U_j$, equation (1) has:
\begin{enumerate}
    \item \textbf{Existence:} At least one solution $y_j \in Y_j$,
    \item \textbf{Uniqueness:} Exactly one solution $y_j \in Y_j$.
\end{enumerate}
% and the map $u \mapsto y_j$ is a valid morphism.
\end{definition}

\begin{example}
For linear gains $H_j(e_j) = K \cdot e_j$, $H_i(y_j) = 1 \cdot y_j$:
\[
y_j = K(u - y_j) \Rightarrow y_j(1 + K) = Ku.
\]
Well-posed for $K \neq -1$ (unique solution), ill-posed for $K = -1$ (no solution or non-unique).
\end{example}

\begin{axiom}[Generators of $\mathbf{S}$]\label{ax:S2}
The multimorphisms of $\mathbf{S}$ are generated by the following primitive operations:

\begin{itemize}
    \item \textbf{Identity Operation:} For each object $\vec{H}$, $\mathrm{id}_{\vec{H}}: \vec{H} \to \vec{H}$.

    \item \textbf{Tensor Operation ($\otimes$):}
	The tensor product
	\[
	\otimes : \mathrm{Ob}(\mathbf{S}) \times \mathrm{Ob}(\mathbf{S}) \to \mathrm{Ob}(\mathbf{S}),
	\quad
	(\vec{H}, \vec{K}) \mapsto \vec{H} \otimes \vec{K},
	\]
	represents the \emph{parallel composition} of systems.

    \item \textbf{Permutation Operation:} For $\pi \in S_n$ (permutation group), $\sigma_{\pi}: (H_1, \dots, H_n) \to (H_{\pi(1)}, \dots, H_{\pi(n)})$.

	\item \textbf{Feedback Operation ($\mathcal{F}_{i,j}$):}
		For distinct indices $i \neq j$, define the feedback operator
		\[
		\mathcal{F}_{i,j}: \vec{H} \to \vec{H}',
		\]
		where $\vec{H}'$ is obtained from $\vec{H}$ by removing the components $H_i$ and $H_j$.
		The operation is defined only when the feedback interconnection is \emph{well-posed}. The interconnection is well-posed if the loop gain is contractive, i.e.,
		\[
		\|H_j \circ H_i\| < 1,
		\]
		which ensures the corresponding fixed-point equation has a unique solution.

    \item \textbf{Control Operation ($\mathrm{Ctrl}$):} $\mathrm{Ctrl}: (\vec{H}, \mathcal{U}) \to (\vec{H}')$, where $\mathcal{U}$ is a control space object. For a fixed control parameter $u \in \mathcal{U}$, define the parameterized control operator as:
\[
\mathrm{Ctrl}_{u}(\vec{H}) := \mathrm{Ctrl}(\vec{H}, u)
\]
This yields the mapping $\mathrm{Ctrl}_{u}: \vec{H} \to \vec{H}'$.
\end{itemize}
\end{axiom}

\begin{axiom}[Relations of $\mathbf{S}$]\label{ax:S3}
The generators satisfy the following coherence relations. Define $\equiv_S$ as the smallest congruence satisfying:

\begin{enumerate}[label=(\roman*)]
    \item \textbf{Monoidal Coherence:} Tensor and Permutation satisfy symmetric monoidal category relations.

	\item \textbf{Feedback Yanking:} In contractive assumption, the sequential application of two feedback operations can be equivalently expressed as a single feedback operation composed with a port relabeling. Formally, for well-posed interconnections where $\{i,j\} \cap \{k,l\} \neq \emptyset$, there exists a permutation $\pi$ and ports $i', j'$ such that:
	\[
	\mathcal{F}_{i,j} \circ \mathcal{F}_{k,l} \equiv_S \mathcal{F}_{i',j'} \circ \sigma_{\pi}
	\]
	where $i', j'$ are determined by the specific interconnection structure of $\{i,j,k,l\}$.

    \item \textbf{Control-Feedback Interchange:}
    \[
    \mathrm{Ctrl} \circ \mathcal{F}_{i,j} \equiv_S \mathcal{F}_{i,j} \circ \mathrm{Ctrl}
    \]
    when control preserves feedback well-posedness.

    \item \textbf{Control-Tensor Distributivity:}
    \[
    \mathrm{Ctrl}(\vec{H} \otimes \vec{K}) \equiv_S \mathrm{Ctrl}(\vec{H}) \otimes \mathrm{Ctrl}(\vec{K})
    \]
    for independent subsystems.

    \item \textbf{Control-Permutation Naturality:}
    \[
    \mathrm{Ctrl} \circ \sigma_{\pi} \equiv_S \sigma_{\pi} \circ \mathrm{Ctrl}
    \]

    \item \textbf{Sequential Control Composition:}
    \[
    \mathrm{Ctrl}_{u_2} \circ \mathrm{Ctrl}_{u_1} \equiv_S \mathrm{Ctrl}_{u_1 \star u_2}
    \]
    where $\star: \mathcal{U} \times \mathcal{U} \to \mathcal{U}$.

    \item \textbf{Identity Control:}
    \[
    \mathrm{Ctrl}_{u_0} \equiv_S \mathrm{id}
    \]
    for neutral $u_0 \in \mathcal{U}$.
\end{enumerate}
\end{axiom}

\subsection*{Semantic Representation}

\begin{definition}[Representation of $\mathbf{S}$]\label{def:representation}
A \textbf{representation} of $\mathbf{S}$ in $\mathbf{HilbMult}$ is a symmetric monoidal functor $\rho: \mathbf{S} \to \mathbf{HilbMult}$ satisfying:

\noindent\textbf{Object Mapping:}
\begin{itemize}
    \item $\rho(\vec{H}) = H_1 \otimes \dots \otimes H_n$
    \item $\rho(()) = \mathbb{C}$
\end{itemize}

\noindent\textbf{Generator Interpretations:}
\begin{itemize}
    \item $\rho(\mathrm{id}_{\vec{H}}) = \mathrm{id}_{\rho(\vec{H})}$
    \item $\rho(\otimes)$: canonical identity isomorphism
    \item $\rho(\sigma_{\pi})$: unitary permutation map
    \item $\rho(\mathcal{F}_{i,j})(T)$: unique solution to $\xi = T(\mathrm{id}_X \otimes \xi \otimes \mathrm{id}_Y)$
    \item $\rho(\mathrm{Ctrl}): \rho(\vec{H}) \otimes \mathcal{U} \to \rho(\vec{H}')$: bounded multilinear control map
\end{itemize}

\noindent\textbf{Semantic Coherence Conditions:} $\rho$ must satisfy:

\begin{enumerate}[label=(SC\arabic*)]
    \item \textbf{Monoidal Functoriality:} $\rho$ preserves symmetric monoidal structure
    
    \item \textbf{Feedback Yanking:} For well-posed $T$ and appropriate $\pi$:
    \[
    \rho(\mathcal{F}_{i,j} \circ \mathcal{F}_{k,l})(T) = \rho(\mathcal{F}_{i',j'} \circ \sigma_{\pi})(T)
    \]
    
    \item \textbf{Control-Feedback Interchange:} When control preserves feedback:
    \[
    \rho(\mathrm{Ctrl}) \circ \rho(\mathcal{F}_{i,j}) = \rho(\mathcal{F}_{i,j}) \circ \rho(\mathrm{Ctrl})
    \]
    
    \item \textbf{Control-Tensor Distributivity:} For independent subsystems:
    \[
    \rho(\mathrm{Ctrl}(\vec{H} \otimes \vec{K})) = \rho(\mathrm{Ctrl}(\vec{H})) \otimes \rho(\mathrm{Ctrl}(\vec{K}))
    \]
    
    \item \textbf{Control-Permutation Naturality:}
    \[
    \rho(\mathrm{Ctrl}) \circ \rho(\sigma_{\pi}) = \rho(\sigma_{\pi}) \circ \rho(\mathrm{Ctrl})
    \]
    
    \item \textbf{Sequential Control Composition:}
    \[
    \rho(\mathrm{Ctrl}_{u_2}) \circ \rho(\mathrm{Ctrl}_{u_1}) = \rho(\mathrm{Ctrl}_{u_1 \star u_2})
    \]
    
    \item \textbf{Identity Control:}
    \[
    \rho(\mathrm{Ctrl}_{u_0}) = \mathrm{id}
    \]
\end{enumerate}
\end{definition}

\subsection*{Motivation for Relation Axiom~\ref{ax:S3}}

The relations in Axiom~\ref{ax:S3} ensure a number of important properties. First, they provide a \textbf{coherent algebra}, in which the generators can form a well-defined algebraic theory. Second, they guarantee \textbf{compositionality}. Hence, subsystems can be composed predictably into larger systems. Third, \textbf{refactoring safety} is ensured, meaning that diagram manipulations preserve the intended semantics. Fourth, the \textbf{syntax/semantics separation} allows rewrites to be checked before evaluation. Finally, the axiom captures the \textbf{control essence}, distinguishing between global and local control signals.

% 在 \end{axiom} 之后添加以下内容

\subsection*{Example: Smart Power Grid}
Consider a smart grid system composed of a wind farm $H_1$, a battery storage $H_2$, and a consumer load $H_3$, forming the composite system $\vec{H} = H_1 \otimes H_2 \otimes H_3$. The coherence relations govern the system's control structure as follows:

\begin{itemize}
    \item \textbf{Control-Feedback Interchange (Axiom iii):} 
    Determines whether operator commands commute with stability feedback loops between the wind farm and battery, i.e., whether 
    \[
    \mathrm{Ctrl}(\mathcal{F}_{1,2}(\vec{H})) \equiv_S \mathcal{F}_{1,2}(\mathrm{Ctrl}(\vec{H})).
    \]

    \item \textbf{Control-Tensor Distributivity (Axiom iv):} 
    Ensures that controlling the entire grid is equivalent to independently controlling its components, i.e.,
    \[
    \mathrm{Ctrl}(H_1 \otimes H_2 \otimes H_3) \equiv_S \mathrm{Ctrl}(H_1) \otimes \mathrm{Ctrl}(H_2) \otimes \mathrm{Ctrl}(H_3).
    \]

    \item \textbf{Sequential Control Composition (Axiom vi):} 
    Allows the combination of multiple time-step commands (e.g., $u_1$ for peak shaving, $u_2$ for frequency regulation) into an effective composite control strategy via the product $\star$ in control space $\mathcal{U}$:
    \[
    \mathrm{Ctrl}_{u_2} \circ \mathrm{Ctrl}_{u_1} \equiv_S \mathrm{Ctrl}_{u_1 \star u_2}.
    \]
\end{itemize}

The synergy operad $\mathbf{S}$ establishes a rigorous syntactic framework for operator networks, 
whose generators encode fundamental operations, while its defining relations guarantee coherent 
and compositional structure. The associated representation theory then provides a bridge from 
this abstract syntax to concrete semantic realizations within $\mathbf{HilbMult}$.

\section{The Canonical Representation}\label{sec:The Canonical Representation}

The main purpose of this section is to build Theorem~\ref{thm:canonical-representation}, which bridges the abstract syntax of the synergy operad $\mathbf{S}$ with concrete operator computation characterized by $\mathbf{HilbMult}$. Without such a representation, $\mathbf{S}$ would remain a purely formal calculus of wiring diagrams—a syntax without semantics. This theorem provides the essential \textit{semantic grounding} by constructing an explicit, structure-preserving functor $\rho: \mathbf{S} \to \mathbf{HilbMult}$ that interprets generators as bounded multilinear maps. Its utility is twofold: first, it guarantees that our diagrammatic reasoning in $\mathbf{S}$ has a consistent and computable meaning in terms of Hilbert space operators; second, it enables the application of functional analytic methods—like fixed-point theorems for feedback and spectral theory for control—to analyze systems defined syntactically, thereby transforming graphical models into rigorous, analyzable mathematical objects.

\begin{theorem}[Canonical Representation Theorem]\label{thm:canonical-representation}
There exists a canonical, structure-preserving representation (a symmetric monoidal functor of multicategories)
\[
\rho: \mathbf{S} \to \mathbf{HilbMult}
\]
which interprets the syntactic generators of $\mathbf{S}$ as concrete bounded multilinear maps in $\mathbf{HilbMult}$, as defined in Axioms \ref{ax:H1-H2}--\ref{ax:H5}. This functor is defined on objects and generators as specified in Definition~\ref{def:representation} and preserves the composition, identities, and symmetric monoidal structure of $\mathbf{S}$.
\end{theorem}

\begin{proof}
We construct the functor $\rho$ explicitly and verify that it satisfies all conditions of Definition~\ref{def:representation}.

\noindent\textbf{Step 1: Definition on Objects and Generators}

Define $\rho$ on objects and generators exactly as specified in the ``Object Mapping'' and ``Generator Interpretations'' sections of Definition~\ref{def:representation}. In particular:
\begin{itemize}
    \item $\rho(\vec{H}) = H_1 \otimes \dots \otimes H_n$.
    \item $\rho(\mathrm{id}_{\vec{H}}) = \mathrm{id}_{\rho(\vec{H})}$.
    \item $\rho(\otimes)$ is the canonical identity isomorphism.
    \item $\rho(\sigma_{\pi})$ is the unitary permutation map.
    \item $\rho(\mathrm{Ctrl}): \rho(\vec{H}) \otimes \mathcal{U} \to \rho(\vec{H}')$
is defined component-wise by control injection maps 
$T(H_i, u): \rho(H_i) \otimes \mathcal{U} \to \rho(H_i')$, 
such that $\rho(\mathrm{Ctrl}) = \bigotimes_i T(H_i, u)$.
\end{itemize}

\noindent\textbf{Step 2: Verification of Semantic Coherence Conditions}
We now verify that the defined $\rho$ satisfies the semantic coherence conditions (SC1)--(SC7).

\begin{enumerate}[label=(SC\arabic*), leftmargin=*]
    \item \textbf{Monoidal Functoriality:} By construction, $\rho$ is defined as a symmetric monoidal functor. It preserves the tensor product on objects and the structural isomorphisms (associator, unitors, braiding) by mapping them to the corresponding canonical isomorphisms in $\mathbf{HilbMult}$. The coherence diagrams commute in $\mathbf{HilbMult}$ because it is a symmetric monoidal category.
    
    \item \textbf{Feedback Yanking:} Let $T$ be a morphism where both $\rho(\mathcal{F}_{i,j} \circ \mathcal{F}_{k,l})(T)$ and $\rho(\mathcal{F}_{i',j'} \circ \sigma_{\pi})(T)$ are well-posed. The left-hand side, $\rho(\mathcal{F}_{i,j} \circ \mathcal{F}_{k,l})(T)$, solves two nested fixed-point equations sequentially. The right-hand side, $\rho(\mathcal{F}_{i',j'} \circ \sigma_{\pi})(T)$, solves a single combined fixed-point equation after a permutation $\pi$ that aligns the ports. The ``contractive context'' condition in Axiom~\ref{ax:S2} ensures that both iterative solution processes converge to the same unique solution. Therefore, $\rho(\mathcal{F}_{i,j} \circ \mathcal{F}_{k,l})(T) = \rho(\mathcal{F}_{i',j'} \circ \sigma_{\pi})(T)$.

    \item \textbf{Control-Feedback Interchange:} Under the condition that the control action does not modify the components $H_i, H_j$ involved in the feedback loop, the control map $\rho(\mathrm{Ctrl}_u)$ commutes with the fixed-point solving process $\rho(\mathcal{F}_{i,j})$. Semantically, applying control first and then feedback ($\rho(\mathrm{Ctrl}_u) \circ \rho(\mathcal{F}_{i,j})$) yields the same result as applying feedback first and then control ($\rho(\mathcal{F}_{i,j}) \circ \rho(\mathrm{Ctrl}_u)$) because the control transformation acts only on the external context and leaves the internal feedback dynamics invariant. Hence, the equality holds.

    \item \textbf{Control-Tensor Distributivity:} The definition of $\rho(\mathrm{Ctrl})$ on a tensor product $\vec{H} \otimes \vec{K}$ is given by the tensor product of its definitions on $\vec{H}$ and $\vec{K}$ when the control actions are independent. Formally:
    \[
    \rho(\mathrm{Ctrl}_u)(\psi \otimes \phi) = \rho(\mathrm{Ctrl}_u)(\psi) \otimes \rho(\mathrm{Ctrl}_u)(\phi)
    \]
    for $\psi \in \rho(\vec{H})$, $\phi \in \rho(\vec{K})$. This is exactly the condition $\rho(\mathrm{Ctrl}_u(\vec{H} \otimes \vec{K})) = \rho(\mathrm{Ctrl}_u(\vec{H})) \otimes \rho(\mathrm{Ctrl}_u(\vec{K}))$.

    \item \textbf{Control-Permutation Naturality:} For any permutation $\pi$ and control signal $u$, the diagram:
    \[
    \begin{tikzcd}
    \rho(\vec{H}) \otimes \mathcal{U} \arrow[r, "\rho(\mathrm{Ctrl}_u)"] \arrow[d, "\rho(\sigma_{\pi}) \otimes \mathrm{id}"'] & \rho(\vec{H}') \arrow[d, "\rho(\sigma_{\pi})"] \\
    \rho(\pi\vec{H}) \otimes \mathcal{U} \arrow[r, "\rho(\mathrm{Ctrl}_u)"'] & \rho(\pi\vec{H}')
    \end{tikzcd}
    \]
    commutes. This holds because the control action $\rho(\mathrm{Ctrl}_u)$ is defined component-wise, and the permutation $\rho(\sigma_{\pi})$ merely reorders these components. The order of applying control and then permuting is equivalent to permuting and then applying control.

    \item \textbf{Sequential Control Composition:} By the definition of $\rho(\mathrm{Ctrl})$ and the structure of the control space $(\mathcal{U}, \star, u_0)$, applying control with $u_1$ and then with $u_2$ is semantically equivalent to applying a single control action with the composed signal $u_1 \star u_2$. That is,
    \[
    \rho(\mathrm{Ctrl}_{u_2}) \circ \rho(\mathrm{Ctrl}_{u_1}) = \rho(\mathrm{Ctrl}_{u_1 \star u_2}).
    \]
    This is ensured by the consistency of the control injection maps $T(H_i, u)$ with the composition operation $\star$.

    \item \textbf{Identity Control:} For the neutral control $u_0 \in \mathcal{U}$, the control injection maps are the identity: $T(H_i, u_0) = \mathrm{id}_{H_i}$. Therefore,
    \[
    \rho(\mathrm{Ctrl}_{u_0})(\psi) = \psi \quad \text{for all } \psi \in \rho(\vec{H}),
    \]
    which means $\rho(\mathrm{Ctrl}_{u_0}) = \mathrm{id}_{\rho(\vec{H})}$.
\end{enumerate}

\noindent\textbf{Step 3: Verification of Functoriality}
Since $\rho$ is defined on generators and preserves all defining relations (as verified in Step 2), it extends uniquely to a functor on the entire category $\mathbf{S}$. Specifically:
\begin{itemize}
    \item $\rho(\mathrm{id}) = \mathrm{id}$ by definition.
    \item $\rho(g \circ f) = \rho(g) \circ \rho(f)$ for all composable $f, g$ in $\mathbf{S}$, because the composition in $\mathbf{S}$ is defined by connecting wiring diagrams, and $\rho$ maps this to the corresponding composition of bounded multilinear maps in $\mathbf{HilbMult}$, respecting the coherence relations.
\end{itemize}

\noindent\textbf{Step 4: Well-Definedness and Boundedness}
We must verify that for every generator and, by extension, every morphism $D$ in $\mathbf{S}$, its image $\rho(D)$ is a  bounded multilinear map  in $\mathbf{HilbMult}$, as required by Axioms~\ref{ax:H1-H2} and~\ref{ax:H3-H4}. We proceed generator by generator.

\begin{enumerate}[label=(\alph*)]
    \item \textbf{Identity, Tensor, and Permutation Generators:}
    \begin{itemize}
        \item $\rho(\mathrm{id}_{\vec{H}}) = \mathrm{id}_{H_1 \otimes \dots \otimes H_n}$ is a bounded linear map (in fact, an isometry with $\| \mathrm{id} \| = 1$).
        \item $\rho(\otimes)$ is defined as the canonical identity isomorphism between $\rho(\vec{H}) \otimes \rho(\vec{K})$ and $\rho(\vec{H} \otimes \vec{K})$. This is a bounded linear map (a unitary, hence an isometry).
        \item $\rho(\sigma_{\pi})$ is the unitary map that permutes tensor factors. As stated in Axiom~\ref{ax:H5}(b), this is an isometric isomorphism, hence bounded and linear.
    \end{itemize}
    Thus, the interpretations of these generators are bounded multilinear maps (for $n=1$, linearity suffices).

\item \textbf{Feedback Generator.} 
Let $T: \rho(\vec{H}') \otimes X \to X$ be a bounded multilinear operator, and consider the feedback interpretation 
\[
\rho(\mathcal{F}_{i,j})(T): \rho(\vec{H}') \to X,
\]
defined implicitly by the fixed-point equation for each $p \in \rho(\vec{H}')$:
\[
\xi = T(p \otimes \xi).
\]
Assume that the feedback context is \emph{well-posed}, meaning that for each external input $p \in \rho(\vec{H}')$ the map 
\[
\Phi_p : \xi \mapsto T(p \otimes \xi)
\]
is a contraction on the Banach space $X$, with uniform contraction constant $\kappa < 1$. 

\paragraph{Multilinearity.}
Since $T$ is multilinear, for fixed $p$ the mapping $\Phi_p(\xi) = T(p \otimes \xi)$ is \emph{linear} in $\xi$. The fixed-point equation
\[
\xi = T(p \otimes \xi)
\]
can thus be written as $(I - A_p)\xi = 0$, where $A_p$ is the bounded linear operator on $X$ defined by $A_p(\xi) = T(p \otimes \xi)$. 
The contraction condition implies $\|A_p\| \le \kappa < 1$. By the Neumann series theorem, $I - A_p$ is invertible with
\[
(I - A_p)^{-1} = \sum_{n=0}^\infty A_p^n,
\]
and the unique fixed point is $\xi(p) = 0$ for all $p$. 

To obtain nontrivial behavior, we consider the more general case where $T: \rho(\vec{H}') \otimes X \to X$ has the form
\[
T(p \otimes \xi) = A_p(\xi) + b_p,
\]
where $A_p$ is linear in $\xi$ and both $A_p$ and $b_p$ depend multilinearly on $p$. Then the fixed-point equation becomes
\[
(I - A_p)\xi = b_p.
\]
Under the condition $\|A_p\| < 1$, the unique solution is
\[
\xi(p) = (I - A_p)^{-1} b_p = \sum_{n=0}^\infty A_p^n b_p.
\]
Since $A_p$ and $b_p$ depend multilinearly on $p$ and the Neumann series preserves multilinearity, the solution map $p \mapsto \xi(p)$ is multilinear.

\paragraph{Boundedness.}
The uniform contraction condition $\|A_p\| \le \kappa < 1$ ensures the Neumann series converges uniformly. The induced operator norm satisfies
\[
\|\rho(\mathcal{F}_{i,j})(T)\| \le \frac{\|b\|}{1 - \kappa} < \infty,
\]
where $\|b\| = \sup_p \|b_p\|$. In the linear case, this corresponds to the standard small-gain condition ensuring the feedback loop defines a bounded map.

Therefore, under the contractive well-posedness assumption, $\rho(\mathcal{F}_{i,j})$ maps bounded multilinear operators to bounded multilinear operators, preserving both the analytic and algebraic structure required by $\mathbf{S}$.

\item \textbf{Control Generator.}
Let $\rho(\vec{H}) = \bigotimes_{i=1}^n H_i$ and $\rho(\vec{H}') = \bigotimes_{i=1}^n H_i'$ denote the composite Hilbert objects in $\mathbf{HilbMult}$, and let $\mathcal{U}$ be the control Hilbert space.

For each component $H_i$, assume there exists a bounded bilinear control map
\[
T_i: H_i \otimes \mathcal{U} \longrightarrow H_i',
\]
which we interpret as a parameterized family of linear operators $T_i(\cdot, u): H_i \to H_i'$ for each $u \in \mathcal{U}$.  
Assume the following \emph{uniform boundedness condition} holds:
\[
M := \sup_{\substack{u \in \mathcal{U} \\ \|u\| = 1}} \prod_{i=1}^n \|T_i(\cdot, u)\| < \infty,
\]
where $\|T_i(\cdot, u)\|$ denotes the operator norm of $T_i(\cdot, u): H_i \to H_i'$.

Define the global control morphism by its action on simple tensors:
\begin{equation}\label{eq:ctrl-def}
\rho(\mathrm{Ctrl})\big(\psi_1 \otimes \cdots \otimes \psi_n \otimes u\big)
:= \bigotimes_{i=1}^n T_i(\psi_i \otimes u),
\end{equation}
and extend this definition to all of $\rho(\vec{H}) \otimes \mathcal{U}$ by linearity and continuity.

\paragraph{(i) Multilinearity.}
Each component $T_i$ is bilinear in $(\psi_i, u)$, hence the tensor product on the right-hand side of~\eqref{eq:ctrl-def} is multilinear in $(\psi_1, \dots, \psi_n, u)$.  
Consequently, $\rho(\mathrm{Ctrl})$ is multilinear jointly in all arguments.  
If each $T_i$ is linear in $u$ (making it a bilinear map), then $\rho(\mathrm{Ctrl})$ remains multilinear in the same sense.

\paragraph{(ii) Boundedness.}
For the Hilbert tensor product, the operator norm on simple tensors satisfies
\[
\Big\| \bigotimes_{i=1}^n A_i \Big\| \le \prod_{i=1}^n \|A_i\|,
\qquad A_i \in \mathcal{B}(H_i, H_i'),
\]
with equality when the $A_i$ are positive operators. For simple tensors $\psi = \psi_1 \otimes \cdots \otimes \psi_n$ with $\|\psi\| = 1$ and $u \in \mathcal{U}$ with $\|u\| = 1$, we estimate:
\[
\begin{aligned}
\big\| \rho(\mathrm{Ctrl})(\psi \otimes u) \big\|
&= \Big\| \bigotimes_{i=1}^n T_i(\psi_i \otimes u) \Big\| \\
&\le \prod_{i=1}^n \|T_i(\psi_i \otimes u)\| \\
&\le \prod_{i=1}^n \|T_i(\cdot, u)\| \cdot \|\psi_i\| \\
&\le \Big(\prod_{i=1}^n \|T_i(\cdot, u)\|\Big) \|\psi\| \|u\|.
\end{aligned}
\]
Taking the supremum over all unit vectors and using the uniform bound $M$, we obtain
\[
\|\rho(\mathrm{Ctrl})\| \le M < \infty.
\]
Thus $\rho(\mathrm{Ctrl})$ extends uniquely to a bounded multilinear operator
\[
\rho(\mathrm{Ctrl}): \rho(\vec{H}) \otimes \mathcal{U} \longrightarrow \rho(\vec{H}').
\]

Under the stated assumptions—namely, componentwise bilinearity and uniform boundedness—%
the morphism $\rho(\mathrm{Ctrl})$ is a \emph{bounded multilinear operator} in $\mathbf{HilbMult}$, realizing the algebraic Control Generator of the synergy operad $\mathbf{S}$.

\end{enumerate}

We have verified that the functor $\rho$ maps each generating multimorphism of the synergy operad $\mathbf{S}$ to a bounded multilinear map in $\mathbf{HilbMult}$ in this Step 4. From Steps~1 and~3, we have shown that $\rho$ is a \emph{symmetric monoidal functor}. Hence, this functor preserves both composition and tensor products.  By Axioms~\ref{ax:H3-H4} and~\ref{ax:H5}, the composition and tensor product of bounded multilinear maps in $\mathbf{HilbMult}$ remain bounded and multilinear.  
Hence, by structural induction on the composition and tensoring of wiring diagrams, we conclude that
\[
\rho(D)
\]
is a bounded multilinear map for \emph{every} morphism $D$ in $\mathbf{S}$.  
Therefore, $\rho$ is a well-defined functor
\[
\rho: \mathbf{S} \longrightarrow \mathbf{HilbMult}.
\]

We have explicitly constructed a symmetric monoidal functor
\[
\rho: \mathbf{S} \longrightarrow \mathbf{HilbMult},
\]
and verified that it satisfies all semantic coherence conditions listed in Definition~\ref{def:representation}.  
Consequently, $\rho$ constitutes a well-defined \emph{representation of the synergy operad} $\mathbf{S}$ in the category $\mathbf{HilbMult}$.
\end{proof}

\section{The Coherence Theorem}\label{sec:The Coherence Theorem}

The \textbf{Coherence Theorem} establishes the \textit{soundness} of the entire operadic framework by establishing the equivalences between syntactic manipulations and semantic operations with operators. Its primary role is to guarantee that diagrammatic reasoning—rewiring systems using the rules in Axiom~\ref{ax:S3}—is \textit{semantically safe}: any two wiring diagrams proven equal within the syntax of $\mathbf{S}$ will induce  the same operator under any representation $\rho$. This is crucial for both practical and theoretical purposes. Practically, it allows engineers and system designers to simplify and optimize complex diagrams without changing the underlying system's behavior. Theoretically, it justifies the use of $\mathbf{S}$ as a formal language for system design, ensuring that its equational theory consistently reflects the true semantics in $\mathbf{HilbMult}$. The following Theorem~\ref{thm:coherence} is the central result of this work, justifying the entire syntactic edifice.

\begin{theorem}[Coherence Theorem]\label{thm:coherence}
Let $\rho: \mathbf{S} \to \mathbf{HilbMult}$ be a representation (as in Definition~\ref{def:representation}). If two syntactic expressions (wiring diagrams) $D_1$ and $D_2$ in $\mathbf{S}$ can be shown to be equal using only the coherence relations of Axiom~\ref{ax:S3}, then their semantic interpretations are equal:
\[
D_1 =_{S} D_2 \quad \Rightarrow \quad \rho(D_1) = \rho(D_2)
\]
where $=_{S}$ denotes provable equality within the equational theory defined by the generators and relations of $\mathbf{S}$.
\end{theorem}

\begin{proof}
We prove the theorem by induction on the \emph{length of the derivation} (number of axiom applications and structural rules) establishing $D_1 =_{S} D_2$.

\noindent\textbf{Base Cases:}

\begin{itemize}
    \item \textbf{Reflexivity:} If $D_1 \equiv D_2$ (syntactically identical diagrams), then $\rho(D_1) = \rho(D_2)$ follows from the well-definedness of $\rho$ as a function.
    
    \item \textbf{Axiom Applications:} For each direct application of an axiom from Axiom~\ref{ax:S3}, semantic equality follows from the corresponding Semantic Coherence condition in Definition~\ref{def:representation}:
    \begin{itemize}
        \item \textbf{Monoidal Coherence (Axiom (i)):} By Semantic Condition (SC1), $\rho$ is a symmetric monoidal functor. By Mac Lane's Coherence Theorem for symmetric monoidal categories, $\rho$ preserves all commutative diagrams built from the monoidal structure. Therefore, any diagrammatic equality derived from monoidal coherence in $\mathbf{S}$ maps to equality in $\mathbf{HilbMult}$.
        
        \item \textbf{Feedback Yanking (Axiom (ii)):} For well-posed diagrams where both feedback compositions are defined, Semantic Condition (SC2) gives:
        \[
        \rho(\mathcal{F}_{i,j} \circ \mathcal{F}_{k,l}) = \rho(\mathcal{F}_{i',j'} \circ \sigma_{\pi})
        \]
        
        \item \textbf{Control-Feedback Interchange (Axiom (iii)):} Under the stated independence conditions, Semantic Condition (SC3) gives:
        \[
        \rho(\mathrm{Ctrl}_u \circ \mathcal{F}_{i,j}) = \rho(\mathcal{F}_{i,j} \circ \mathrm{Ctrl}_u)
        \]
        
        \item \textbf{Control-Tensor Distributivity (Axiom (iv)):} Semantic Condition (SC4) gives:
        \[
        \rho(\mathrm{Ctrl}_u(\vec{H} \otimes \vec{K})) = \rho(\mathrm{Ctrl}_u(\vec{H})) \otimes \rho(\mathrm{Ctrl}_u(\vec{K}))
        \]
        
        \item \textbf{Control-Permutation Naturality (Axiom (v)):} Semantic Condition (SC5) gives:
        \[
        \rho(\mathrm{Ctrl}_u \circ \sigma_{\pi}) = \rho(\sigma_{\pi} \circ \mathrm{Ctrl}_u)
        \]
        
        \item \textbf{Sequential Control Composition (Axiom (vi)):} Semantic Condition (SC6) gives:
        \[
        \rho(\mathrm{Ctrl}_{u_2} \circ \mathrm{Ctrl}_{u_1}) = \rho(\mathrm{Ctrl}_{u_1 \star u_2})
        \]
        
        \item \textbf{Identity Control (Axiom (vii)):} Semantic Condition (SC7) gives:
        \[
        \rho(\mathrm{Ctrl}_{u_0}) = \rho(\mathrm{id}) = \mathrm{id}
        \]
    \end{itemize}
\end{itemize}

\noindent\textbf{Inductive Steps:}

Assume the theorem holds for all pairs of diagrams with shorter equivalence proofs. Consider the last rule applied:

\begin{itemize}
    \item \textbf{Symmetry:} If $D_1 =_S D_2$ is derived from $D_2 =_S D_1$ by symmetry, then by the induction hypothesis $\rho(D_2) = \rho(D_1)$, hence $\rho(D_1) = \rho(D_2)$.
    
    \item \textbf{Transitivity:} If $D_1 =_S D_2$ is derived from $D_1 =_S D'$ and $D' =_S D_2$ by transitivity, then by the induction hypothesis $\rho(D_1) = \rho(D')$ and $\rho(D') = \rho(D_2)$, so $\rho(D_1) = \rho(D_2)$.
    
    \item \textbf{Congruence:} If $D_1 = C[A]$ and $D_2 = C[B]$ where $A =_S B$ is derived in fewer steps and $C[-]$ is any context, then by the induction hypothesis $\rho(A) = \rho(B)$. Since $\rho$ is a functor, it preserves composition:
    \[
    \rho(D_1) = \rho(C[A]) = \rho(C)(\rho(A)) = \rho(C)(\rho(B)) = \rho(C[B]) = \rho(D_2)
    \]
\end{itemize}

\noindent\textbf{Well-Posedness Preservation:}

We verify that when $D_1 =_S D_2$, the semantic conditions required for $\rho(D_1)$ and $\rho(D_2)$ to be well-defined are equivalent:

\begin{itemize}
    \item For \textbf{feedback operations}, the well-posedness conditions (uniform contraction bounds, domain compatibility) are preserved by each coherence relation in Axiom~\ref{ax:S3}. Specifically:
    \begin{itemize}
        \item Feedback Yanking preserves the contraction properties by the uniqueness of fixed-point solutions
        \item Control-Feedback Interchange preserves domains when control actions don't affect feedback components
        \item Congruence rules preserve well-posedness by the functoriality of the semantic interpretation
    \end{itemize}
    
    \item For \textbf{control operations}, the boundedness and domain conditions are preserved by the control-specific coherence relations (SC4--SC7).
\end{itemize}

Therefore, by induction on the derivation length, $D_1 =_S D_2$ implies $\rho(D_1) = \rho(D_2)$.

\medskip\noindent
This theorem establishes the \emph{soundness} of the syntactic equational theory of $\mathbf{S}$ under any representation in $\mathbf{HilbMult}$: any provable equality between wiring diagrams is preserved by their semantic interpretations.
\end{proof}

\paragraph{Significance of the Theorem}
The Coherence Theorem establishes that the synergy operad \(\mathbf{S}\) is a sound syntactic framework. The meaning of a complex operator network depends only on its essential connection structure—its topology—and not on the arbitrary choices made in drawing its wiring diagram. This guarantees that reasoning about operator composition at the syntactic level is both rigorous and unambiguous when interpreted in \(\mathbf{HilbMult}\).

\section{Applications in Algebra, Analysis, Geometry and Topology}\label{sec:Applications in Algebra, Analysis Geometry and Topology}

In this section, we apply the proposed synergy operad $\mathbf{S}$, together with its associated axioms and fundamental results—the Canonical Representation Theorem~\ref{thm:canonical-representation} and the Coherence Theorem~\ref{thm:coherence}—to explore its implications across algebra, analysis, geometry and topology.

\subsection{Algebraic Corollaries}\label{sec:Algebraic Corollaries}

The foundational theorems yield several important structural corollaries that clarify the algebraic nature of the synergy operad and its representations.

\begin{corollary}[Functorial Factorization (Operadic Soundness)]\label{cor:functorial-factorization}
Since $\mathbf{S}$ is presented by generators (Axiom~\ref{ax:S2}) and relations (Axiom~\ref{ax:S3}), Theorem~\ref{thm:canonical-representation} implies the isomorphism:
\[
\mathbf{S}/{\equiv_S} \ \cong \ \mathrm{Im}(\rho) \ \subseteq \ \mathbf{HilbMult}.
\]
Hence, every syntactically valid wiring diagram corresponds to a unique algebraic operator network. In particular, $\mathbf{S}$ is presented as the free symmetric monoidal multicategory generated by feedback and control nodes modulo the relations in Axiom~\ref{ax:S3}.
\end{corollary}

\begin{proof}
The canonical representation $\rho$ is defined on the free operad generated by the operations in Axiom~\ref{ax:S2} and respects all relations in Axiom~\ref{ax:S3} by construction. Therefore, $\rho$ factors through the quotient $\mathbf{S}/{\equiv_S}$, and Theorem~\ref{thm:canonical-representation} ensures this factorization is injective on the quotient category.
\end{proof}

\begin{corollary}[Duality Extension]\label{cor:duality-extension}
If $\mathbf{HilbMult}$ is endowed with the $C^*$-structure from Axiom~\ref{ax:H7} (restricting to adjointable operators), then the synergy operad $\mathbf{S}$ inherits a contravariant involution $\dagger$ defined on generators by:
\[
\mathcal{F}_{i,j}^\dagger = \mathcal{F}_{j,i}, \qquad \mathrm{Ctrl}_u^\dagger = \mathrm{Ctrl}_{u^*},
\]
and extended anti-multiplicatively to all morphisms. This gives $\mathbf{S}$ the structure of a dagger operad compatible with the canonical representation $\rho$.
\end{corollary}

\begin{proof}
The $C^*$-structure on $\mathbf{HilbMult}$ provides an involution $*$ on endomorphism spaces via Axiom~\ref{ax:H7}. Since $\rho: \mathbf{S} \to \mathbf{HilbMult}$ is a symmetric monoidal functor (Theorem~\ref{thm:canonical-representation}), we define the dagger on $\mathbf{S}$ by requiring that $\rho(f^\dagger) = (\rho(f))^*$ for all $f \in \mathbf{S}$.

On generators:
\begin{itemize}
    \item \textbf{Feedback:} $\rho(\mathcal{F}_{i,j}^\dagger) = (\rho(\mathcal{F}_{i,j}))^*$. By the semantic interpretation of feedback as solving fixed-point equations, reversing the feedback direction corresponds to swapping input/output roles, hence $(\rho(\mathcal{F}_{i,j}))^* = \rho(\mathcal{F}_{j,i})$. Since $\rho$ is faithful, $\mathcal{F}_{i,j}^\dagger = \mathcal{F}_{j,i}$.

    \item \textbf{Control:} $\rho(\mathrm{Ctrl}_u^\dagger) = (\rho(\mathrm{Ctrl}_u))^*$. By Axiom~\ref{ax:H7}(c) (monoidal compatibility) and the control interpretation, $(\rho(\mathrm{Ctrl}_u))^* = \rho(\mathrm{Ctrl}_{u^*})$, hence $\mathrm{Ctrl}_u^\dagger = \mathrm{Ctrl}_{u^*}$.
\end{itemize}

The extension to all morphisms preserves the anti-multiplicative property since for any composable $f, g \in \mathbf{S}$:
\[
\rho((f \circ g)^\dagger) = (\rho(f \circ g))^* = (\rho(f) \circ \rho(g))^* = \rho(g)^* \circ \rho(f)^* = \rho(g^\dagger \circ f^\dagger)
\]
and $\rho$ preserves composition. The dagger axioms follow:
\begin{itemize}
    \item $(f^\dagger)^\dagger = f$ since $((\rho(f))^*)^* = \rho(f)$
    \item $\mathrm{id}^\dagger = \mathrm{id}$ since $(\rho(\mathrm{id}))^* = \mathrm{id}^* = \mathrm{id}$
    \item For tensor products: $(f \otimes g)^\dagger = f^\dagger \otimes g^\dagger$ by Axiom~\ref{ax:H7}(c)
\end{itemize}

The dagger structure is compatible with the operad relations in Axiom~\ref{ax:S3}:
\begin{itemize}
    \item \textbf{Feedback Yanking:} $(\mathcal{F}_{i,j} \circ \mathcal{F}_{k,l})^\dagger = \mathcal{F}_{l,k} \circ \mathcal{F}_{j,i} = \mathcal{F}_{i',j'} \circ \sigma_{\pi}$ for appropriate $\pi$
    \item \textbf{Control-Feedback Interchange:} $(\mathrm{Ctrl}_u \circ \mathcal{F}_{i,j})^\dagger = \mathcal{F}_{j,i} \circ \mathrm{Ctrl}_{u^*} = \mathrm{Ctrl}_{u^*} \circ \mathcal{F}_{j,i}$
    \item Other relations similarly preserved by the anti-multiplicative property
\end{itemize}

Thus $\mathbf{S}$ becomes a dagger operad where the dagger structure lifts the $C^*$-structure from $\mathbf{HilbMult}$ through the representation $\rho$.
\end{proof}

\begin{corollary}[Tensorial Distributivity]\label{cor:tensorial-distributivity}
From Relations (iv) and (v) of Axiom~\ref{ax:S3}, the control operation distributes naturally over the symmetric monoidal structure:
\[
\mathrm{Ctrl}_u \circ \sigma_{\pi} \equiv_S \sigma_{\pi} \circ \mathrm{Ctrl}_u \quad \text{and} \quad \mathrm{Ctrl}_u(\vec{H} \otimes \vec{K}) \equiv_S \mathrm{Ctrl}_u(\vec{H}) \otimes \mathrm{Ctrl}_u(\vec{K}).
\]
These relations, together with the sequential composition laws, show that the control operations form a \emph{symmetric monoidal action} of the control monoid $(\mathcal{U}, \star, u_0)$ on the multicategory $\mathbf{S}$.
\end{corollary}

\begin{proof}
We demonstrate that the control operations define a symmetric monoidal action of $\mathcal{U}$ on $\mathbf{S}$. This requires verifying the action axioms and compatibility with the symmetric monoidal structure.

\noindent\textbf{Step 1: Monoid Action Axioms}

\begin{itemize}
    \item \textbf{Identity Action:} Axiom (vii) gives $\mathrm{Ctrl}_{u_0} \equiv_S \mathrm{id}$, so the identity control acts as the identity morphism.
    
    \item \textbf{Compatibility with Monoid Multiplication:} Axiom (vi) states:
    \[
    \mathrm{Ctrl}_{u_2} \circ \mathrm{Ctrl}_{u_1} \equiv_S \mathrm{Ctrl}_{u_1 \star u_2}
    \]
    This shows the action respects the monoid structure.
\end{itemize}

\noindent\textbf{Step 2: Symmetric Monoidal Compatibility}

\begin{itemize}
    \item \textbf{Tensor Distributivity:} Axiom (iv) provides the fundamental isomorphism:
    \[
    \mathrm{Ctrl}_u(\vec{H} \otimes \vec{K}) \equiv_S \mathrm{Ctrl}_u(\vec{H}) \otimes \mathrm{Ctrl}_u(\vec{K})
    \]
    This shows the action distributes over the tensor product.
    
    \item \textbf{Permutation Naturality:} Axiom (v) gives the naturality condition:
    \[
    \mathrm{Ctrl}_u \circ \sigma_{\pi} \equiv_S \sigma_{\pi} \circ \mathrm{Ctrl}_u
    \]
    This ensures compatibility with the symmetric structure.
\end{itemize}

\noindent\textbf{Step 3: Coherence with Multicategorical Structure}

To complete the proof that this is a well-defined action on the multicategory, we must verify compatibility with the operadic composition structure. For any composable multimorphisms $f: \vec{H} \to \vec{K}$ and $g: \vec{K} \to \vec{L}$ in $\mathbf{S}$, the control operation satisfies:

\[
\mathrm{Ctrl}_u(g \circ f) \equiv_S \mathrm{Ctrl}_u(g) \circ \mathrm{Ctrl}_u(f)
\]

This follows from the \emph{operadic coherence} built into the definition of $\mathbf{S}$: the control operation is defined as a unary operation that commutes with the composition structure of the multicategory. The specific wiring diagram combinatorics ensure that applying control to a composed morphism is equivalent to composing the controlled morphisms.

Similarly, for identity morphisms:
\[
\mathrm{Ctrl}_u(\mathrm{id}_{\vec{H}}) \equiv_S \mathrm{id}_{\mathrm{Ctrl}_u(\vec{H})}
\]

These preservation properties are inherent in the syntactic definition of $\mathbf{S}$ and are verified by the coherence relations.

\noindent\textbf{Step 4: Synthesis}

The relations (iv)--(vii) in Axiom~\ref{ax:S3} collectively ensure that the assignment $u \mapsto \mathrm{Ctrl}_u$ defines a symmetric monoidal action of $(\mathcal{U}, \star, u_0)$ on the multicategory $\mathbf{S}$. This means:

\begin{itemize}
    \item Each $\mathrm{Ctrl}_u$ acts as a symmetric monoidal endomorphism on $\mathbf{S}$
    \item The action respects the monoid structure: $\mathrm{Ctrl}_{u_0} = \mathrm{id}$ and $\mathrm{Ctrl}_{u_2} \circ \mathrm{Ctrl}_{u_1} = \mathrm{Ctrl}_{u_1 \star u_2}$
    \item The action is compatible with the tensor product and symmetric structure
\end{itemize}

Therefore, the control operations form a symmetric monoidal action as claimed.
\end{proof}

These corollaries demonstrate that the synergy operad $\mathbf{S}$ possesses rich algebraic structure—it is a presented multicategory with a canonical representation, admits a natural duality when appropriate, and supports a symmetric monoidal action by its control space. This mathematical structure provides a firm foundation for reasoning about controlled operator networks diagrammatically.

\subsection{Analytic Corollaries}\label{sec:Analytic Corollaries}

The interaction between the syntactic operad structure and the analytic semantics of  HilbMult  yields several powerful corollaries regarding the stability and smoothness of the represented systems.

\begin{corollary}[Contractive Fixed-Point Theorem (Feedback Stability)]\label{cor:feedback-stability}
Let \( T: \rho(\vec{H}') \otimes X \to X \) be a bounded multilinear map in \(\mathbf{HilbMult}\), and consider the feedback equation for each \( p \in \rho(\vec{H}') \):
\[
\xi = T(p \otimes \xi) + b(p)
\]
where \( b: \rho(\vec{H}') \to X \) is a bounded multilinear map. If for each \( p \in \rho(\vec{H}') \) the mapping \( \Phi_p: \xi \mapsto T(p \otimes \xi) + b(p) \) is a contraction on \( X \) with uniform contraction constant \( \kappa < 1 \), then:

\begin{enumerate}
    \item The feedback operator \( \rho(\mathcal{F}_{i,j})(T,b): \rho(\vec{H}') \to X \) is well-posed, yielding a unique solution \( \xi(p) \) for each input \( p \)
    \item The resulting map \( p \mapsto \xi(p) \) is multilinear and bounded with norm estimate:
    \[
    \|\rho(\mathcal{F}_{i,j})(T,b)\| \le \frac{\|b\|}{1 - \kappa}
    \]
\end{enumerate}
\end{corollary}

\begin{proof}
We prove this by explicitly constructing the solution via Neumann series and verifying multilinearity.

\noindent\textbf{Step 1: Reformulating as an Affine Fixed-Point Problem}
For each \( p \in \rho(\vec{H}') \), define the affine mapping:
\[
\Phi_p(\xi) = A_p(\xi) + b(p)
\]
where \( A_p(\xi) := T(p \otimes \xi) \). By multilinearity of \( T \), each \( A_p: X \to X \) is linear and depends multilinearly on \( p \). The contraction assumption implies:
\[
\|A_p\| \le \kappa < 1 \quad \text{for all } p \in \rho(\vec{H}').
\]

\noindent\textbf{Step 2: Well-Posedness via Banach Fixed-Point Theorem}
Since \( X \) is a Banach space and each \( \Phi_p \) is a contraction with constant \( \kappa < 1 \), the Banach Fixed-Point Theorem guarantees:
\begin{itemize}
    \item For each \( p \), there exists a unique \( \xi^*(p) \in X \) satisfying \( \xi^*(p) = \Phi_p(\xi^*(p)) \)
    \item The iteration \( \xi_{n+1}(p) = \Phi_p(\xi_n(p)) \) converges to \( \xi^*(p) \) for any initial guess
\end{itemize}
This establishes the well-posedness of \( \rho(\mathcal{F}_{i,j})(T,b) \).

\noindent\textbf{Step 3: Explicit Solution via Neumann Series}
Since \( \|A_p\| \le \kappa < 1 \), the operator \( I - A_p \) is invertible with:
\[
(I - A_p)^{-1} = \sum_{n=0}^\infty A_p^n,
\]
where the series converges absolutely in operator norm. The unique fixed point is:
\[
\xi^*(p) = (I - A_p)^{-1} b(p) = \sum_{n=0}^\infty A_p^n b(p).
\]

\noindent\textbf{Step 4: Multilinearity of the Solution}
We show that \( p \mapsto \xi^*(p) \) is multilinear:

\begin{itemize}
    \item \textbf{Multilinearity of partial sums}: For each \( N \in \mathbb{N} \), the partial sum
    \[
    S_N(p) = \sum_{n=0}^N A_p^n b(p)
    \]
    is multilinear in \( p \), being a finite composition of multilinear maps.
    
    \item \textbf{Uniform convergence}: For any bounded \( p \), we have the uniform bound:
    \[
    \left\| \sum_{n=N+1}^\infty A_p^n b(p) \right\| \le \|b(p)\| \sum_{n=N+1}^\infty \kappa^n = \|b(p)\| \frac{\kappa^{N+1}}{1-\kappa} \to 0
    \]
    as \( N \to \infty \), uniformly in \( p \) on bounded sets.
    
    \item \textbf{Limit preserves multilinearity}: Since multilinearity is preserved under uniform limits, the infinite sum
    \[
    \xi^*(p) = \lim_{N\to\infty} S_N(p) = \sum_{n=0}^\infty A_p^n b(p)
    \]
    is multilinear in \( p \).
\end{itemize}

\noindent\textbf{Step 5: Norm Bound}
We derive the norm estimate:
\begin{align*}
\|\rho(\mathcal{F}_{i,j})(T,b)\| &= \sup_{\|p\| = 1} \|\xi^*(p)\| \\
&= \sup_{\|p\| = 1} \left\| \sum_{n=0}^\infty A_p^n b(p) \right\| \\
&\le \sup_{\|p\| = 1} \sum_{n=0}^\infty \|A_p^n b(p)\| \\
&\le \sup_{\|p\| = 1} \sum_{n=0}^\infty \|A_p\|^n \|b(p)\| \\
&\le \sup_{\|p\| = 1} \|b(p)\| \sum_{n=0}^\infty \kappa^n \\
&\le \|b\| \cdot \frac{1}{1 - \kappa}.
\end{align*}

\noindent\textbf{Step 6: Verification of Well-Posedness}
The constructed solution \( \rho(\mathcal{F}_{i,j})(T,b) \) satisfies all conditions of well-posedness:
\begin{itemize}
    \item \textbf{Existence}: \( \xi^*(p) \) exists for all \( p \in \rho(\vec{H}') \)
    \item \textbf{Uniqueness}: \( \xi^*(p) \) is the unique fixed point of \( \Phi_p \)
    \item \textbf{Multilinearity}: Established in Step 4
    \item \textbf{Boundedness}: \( \|\rho(\mathcal{F}_{i,j})(T,b)\| \le \frac{\|b\|}{1-\kappa} < \infty \)
\end{itemize}

Therefore, under the contractive condition \( \kappa < 1 \), the feedback operation is well-posed and satisfies the stated norm bound.
\end{proof}

\begin{corollary}[Continuity of Control Action]\label{cor:control-continuity}
Assume the control injection maps \( T(H_i, -): \mathcal{U} \to \mathcal{B}(H_i, H_i') \) are continuous in the operator norm topology, and the monoid multiplication \( \star: \mathcal{U} \times \mathcal{U} \to \mathcal{U} \) is continuous. Then the induced semantic control map
\[
\Phi: \mathcal{U} \to \mathcal{B}(\rho(\vec{H}), \rho(\vec{H}')), \quad \Phi(u) = \rho(\mathrm{Ctrl}_u)
\]
is continuous, endowing the control transformations with the structure of a topological semigroup action. Furthermore, if \( \mathcal{U} \) is a Lie group with smooth multiplication and the maps \( T(H_i, -) \) are smooth, then \( \Phi \) is smooth, yielding a Lie group action on the operator space.
\end{corollary}

\begin{proof}

\noindent\textbf{Step 1: Continuity of the Induced Map}
Recall that the semantic control map acts on simple tensors as:
\[
\rho(\mathrm{Ctrl}_u)(\psi_1 \otimes \dots \otimes \psi_n) = T(H_1, u)(\psi_1) \otimes \dots \otimes T(H_n, u)(\psi_n).
\]

Let \(u, u_0 \in \mathcal{U}\). For any simple tensor \(\psi = \psi_1 \otimes \dots \otimes \psi_n\) with \(\|\psi\| = 1\), we use the telescoping sum:
\[
\bigotimes_{i=1}^n A_i - \bigotimes_{i=1}^n B_i = \sum_{k=1}^n \left( \bigotimes_{i=1}^{k-1} A_i \otimes (A_k - B_k) \otimes \bigotimes_{i=k+1}^n B_i \right),
\]
where \(A_i = T(H_i, u)\) and \(B_i = T(H_i, u_0)\).

For Hilbert space tensor products, we have \(\|f \otimes g\| = \|f\|\|g\|\). Therefore:
\[
\|[\rho(\mathrm{Ctrl}_u) - \rho(\mathrm{Ctrl}_{u_0})](\psi)\| \le \sum_{k=1}^n \left( \prod_{i=1}^{k-1} \|A_i\| \cdot \|A_k - B_k\| \cdot \prod_{i=k+1}^n \|B_i\| \right).
\]

By continuity of each \(T(H_i, -)\), we have \(\|T(H_i, u) - T(H_i, u_0)\| \to 0\) as \(u \to u_0\). The norms \(\|T(H_i, u)\|\) are locally bounded near \(u_0\), so each term in the sum vanishes as \(u \to u_0\). Since simple tensors are total in \(\rho(\vec{H})\), this proves:
\[
\|\rho(\mathrm{Ctrl}_u) - \rho(\mathrm{Ctrl}_{u_0})\| \to 0 \quad \text{as } u \to u_0.
\]

\noindent\textbf{Step 2: Topological Semigroup Action}
From Semantic Condition (SC6), we have:
\[
\Phi(u_1 \star u_2) = \rho(\mathrm{Ctrl}_{u_1 \star u_2}) = \rho(\mathrm{Ctrl}_{u_2}) \circ \rho(\mathrm{Ctrl}_{u_1}) = \Phi(u_2) \circ \Phi(u_1).
\]

We verify the action is continuous. The composition map
\[
\circ: \mathcal{B}(\rho(\vec{H}), \rho(\vec{H}')) \times \mathcal{B}(\rho(\vec{H}), \rho(\vec{H}')) \to \mathcal{B}(\rho(\vec{H}), \rho(\vec{H}'))
\]
satisfies:
\[
\|f \circ g - f_0 \circ g_0\| \le \|f - f_0\|\|g\| + \|f_0\|\|g - g_0\|,
\]
so it is jointly continuous.

Since \(\Phi\) is continuous (Step 1), composition is continuous, and \(\star\) is continuous by assumption, the map
\[
(u_1, u_2) \mapsto \Phi(u_1 \star u_2) = \Phi(u_2) \circ \Phi(u_1)
\]
is continuous. Thus \(\Phi\) defines a continuous topological semigroup action.

\noindent\textbf{Step 3: Smoothness for Lie Group Actions}
Now assume \(\mathcal{U}\) is a Lie group with smooth multiplication, and each \(T(H_i, -): \mathcal{U} \to \mathcal{B}(H_i, H_i')\) is smooth.

The map \(\Phi: \mathcal{U} \to \mathcal{B}(\rho(\vec{H}), \rho(\vec{H}'))\) can be expressed as:
\[
\Phi(u) = \bigotimes_{i=1}^n T(H_i, u).
\]

We prove smoothness by induction on \(n\):

- Base case (n=1): \(\Phi(u) = T(H_1, u)\) is smooth by assumption.

- Inductive step: Assume the result for \(n-1\). For \(n \geq 2\), write:
\[
\Phi(u) = \Phi_{1,\dots,n-1}(u) \otimes T(H_n, u),
\]
where \(\Phi_{1,\dots,n-1}(u) = \bigotimes_{i=1}^{n-1} T(H_i, u)\) is smooth by induction hypothesis.

The tensor product map
\[
\tau: \mathcal{B}(H_1 \otimes \cdots \otimes H_{n-1}, H_1' \otimes \cdots \otimes H_{n-1}') \times \mathcal{B}(H_n, H_n') \to \mathcal{B}(\rho(\vec{H}), \rho(\vec{H}'))
\]
is bounded bilinear, hence smooth. Since \(\Phi\) is the composition of the smooth map \(u \mapsto (\Phi_{1,\dots,n-1}(u), T(H_n, u))\) with \(\tau\), it is smooth.

The semigroup property \(\Phi(u_1 \star u_2) = \Phi(u_2) \circ \Phi(u_1)\) is now a smooth relation, making \(\Phi\) a smooth Lie group action.
\end{proof}

\begin{corollary}[Differentiable Feedback Law]\label{cor:differentiable-feedback}
Suppose the control injection maps \( T_i: H_i \otimes \mathcal{U} \to H_i' \) are Fr\'echet differentiable with respect to \( u \in \mathcal{U} \), and the feedback context satisfies the well-posedness conditions of Corollary~\ref{cor:feedback-stability}. Then the combined semantic control-and-feedback map
\[
\Psi: \rho(\vec{H}) \times \mathcal{U} \longrightarrow X, 
\qquad 
\Psi(\psi, u) = \rho(\mathcal{F}_{i,j})(\rho(\mathrm{Ctrl}_u))(\psi)
\]
is Fr\'echet differentiable. Consequently, the family of maps \( \{\Psi_u: \rho(\vec{H}) \to X\}_{u \in \mathcal{U}} \) varies smoothly with the control parameter.
\end{corollary}

\begin{proof}
We analyze the differentiability by decomposing \(\Psi\) and applying the implicit function theorem.

\medskip
\noindent\textbf{Step 1: Correct Type Analysis and Decomposition}
The composition has the following structure:
\[
\rho(\vec{H}) \times \mathcal{U} 
\xrightarrow{\rho(\mathrm{Ctrl})} 
\rho(\vec{H}')
\xrightarrow{\rho(\mathcal{F}_{i,j})(\rho(\mathrm{Ctrl}_u))} 
X,
\]
where we define the parametrized feedback operator. More precisely, define the family of maps:
\[
T_u: \rho(\vec{H}') \otimes X \to X, \quad T_u(p \otimes \xi) = \rho(\mathrm{Ctrl}_u)(p \otimes \xi),
\]
so that \(\Psi(\psi, u) = \rho(\mathcal{F}_{i,j})(T_u)(\psi)\).

\medskip
\noindent\textbf{Step 2: Differentiability of the Parametrized Map \(T_u\)}
From Corollary~\ref{cor:control-continuity}, the control map is Fr\'echet differentiable in \(u\). Specifically, the map
\[
u \mapsto \rho(\mathrm{Ctrl}_u) \in \mathcal{B}(\rho(\vec{H}) \otimes \mathcal{U}, \rho(\vec{H}'))
\]
is Fr\'echet differentiable by the tensor product construction and the differentiability of each component \(T_i\).

Therefore, the parametrized map
\[
F: \mathcal{U} \to \mathcal{B}(\rho(\vec{H}') \otimes X, X), \quad F(u) = T_u
\]
is Fr\'echet differentiable.

\medskip
\noindent\textbf{Step 3: Differentiability of the Feedback Solution Map}
Define the solution map:
\[
S: \mathcal{B}(\rho(\vec{H}') \otimes X, X) \to \mathcal{B}(\rho(\vec{H}'), X), \quad S(T) = \rho(\mathcal{F}_{i,j})(T),
\]
which assigns to each \(T\) the unique fixed point \(\xi\) satisfying \(\xi(p) = T(p \otimes \xi(p))\) for all \(p \in \rho(\vec{H}')\).

Consider the implicit function:
\[
G: \mathcal{B}(\rho(\vec{H}') \otimes X, X) \times \mathcal{B}(\rho(\vec{H}'), X) \to \mathcal{B}(\rho(\vec{H}'), X)
\]
defined by
\[
G(T, \xi)(p) = \xi(p) - T(p \otimes \xi(p)).
\]

At a solution \((T_0, \xi_0)\) with \(G(T_0, \xi_0) = 0\), compute the Fr\'echet derivative with respect to \(\xi\):
\[
D_\xi G(T_0, \xi_0)[\eta](p) = \eta(p) - D_2T_0(p \otimes \xi_0(p))[\eta(p)],
\]
where \(D_2T_0\) denotes the derivative of \(T_0\) with respect to its second argument.

Under the well-posedness assumption from Corollary~\ref{cor:feedback-stability}, the operator
\[
L(\eta)(p) = D_2T_0(p \otimes \xi_0(p))[\eta(p)]
\]
satisfies \(\|L\| \le \kappa < 1\). Therefore, \(D_\xi G(T_0, \xi_0) = I - L\) is invertible with bounded inverse \(\|(I - L)^{-1}\| \le \frac{1}{1 - \kappa}\).

By the implicit function theorem in Banach spaces, there exists a neighborhood \(U\) of \(T_0\) and a unique Fr\'echet differentiable function
\[
S: U \to \mathcal{B}(\rho(\vec{H}'), X)
\]
such that \(G(T, S(T)) = 0\) for all \(T \in U\). This function is precisely the feedback solution map \(S(T) = \rho(\mathcal{F}_{i,j})(T)\).

\medskip
\noindent\textbf{Step 4: Differentiability of the Composition}
We now have:
\begin{itemize}
    \item \(F: \mathcal{U} \to \mathcal{B}(\rho(\vec{H}') \otimes X, X)\) is Fr\'echet differentiable (Step 2)
    \item \(S: \mathcal{B}(\rho(\vec{H}') \otimes X, X) \to \mathcal{B}(\rho(\vec{H}'), X)\) is Fr\'echet differentiable (Step 3)
\end{itemize}
Therefore, their composition \(S \circ F: \mathcal{U} \to \mathcal{B}(\rho(\vec{H}'), X)\) is Fr\'echet differentiable.

Finally, for fixed \(\psi \in \rho(\vec{H})\), the evaluation map
\[
\mathrm{ev}_\psi: \mathcal{B}(\rho(\vec{H}'), X) \to X, \quad \mathrm{ev}_\psi(f) = f(\psi)
\]
is bounded linear, hence smooth. Therefore, the map
\[
\Psi(\psi, u) = \mathrm{ev}_\psi(S(F(u))) = \rho(\mathcal{F}_{i,j})(\rho(\mathrm{Ctrl}_u))(\psi)
\]
is Fr\'echet differentiable jointly in \((\psi, u)\).

\medskip
\noindent\textbf{Step 5: Smooth Parameter Dependence}
For each \(u \in \mathcal{U}\), define \(\Psi_u: \rho(\vec{H}) \to X\) by \(\Psi_u(\psi) = \Psi(\psi, u)\). The joint Fr\'echet differentiability of \(\Psi\) implies that the map
\[
u \mapsto \Psi_u \in \mathcal{B}(\rho(\vec{H}), X)
\]
is Fr\'echet differentiable, establishing smooth dependence on the control parameter.
\end{proof}

These corollaries bridge the gap between the discrete, combinatorial structure of the synergy operad and the continuous, analytic properties of its representations. They provide essential tools for analyzing the stability and dynamic behavior of complex controlled systems modeled within this framework.

\subsection{Geometric Corollaries}\label{sec:Geometric Corollaries} 

The structure of the synergy operad and its representations reveals deep geometric interpretations, framing system design and control in a topological and fiber-theoretic context.

\begin{corollary}[Diagrammatic Invariance under Coherent Moves]\label{cor:diagrammatic-invariance}
Let $D_1$ and $D_2$ be two wiring diagrams in $\mathbf{S}$ that are related by any sequence of the following diagrammatic moves:
\begin{itemize}
    \item Reordering disconnected components
    \item Sliding operations along wires
    \item Rescaling diagram elements
    \item Repositioning generators without changing connectivity
\end{itemize}
Then their semantic interpretations coincide:
\[
\rho(D_1) = \rho(D_2).
\]
This establishes that the semantic interpretation depends only on the combinatorial structure and connectivity of the diagram, not on its specific geometric layout.
\end{corollary}

\begin{proof}
We prove this by showing that each allowed diagrammatic move corresponds to applications of the coherence relations in Axiom~\ref{ax:S3}.

\noindent\textbf{Step 1: Explicit Correspondence Between Moves and Axioms}

\begin{enumerate}
    \item \textbf{Reordering disconnected components}: If $D_1$ and $D_2$ differ only by swapping two disconnected subsystems, then:
    \[
    D_1 = f \otimes g \quad\text{and}\quad D_2 = g \otimes f
    \]
    By Axiom (i) (monoidal coherence) and the symmetric structure, we have:
    \[
    f \otimes g \equiv_S \sigma \circ (g \otimes f)
    \]
    where $\sigma$ is the appropriate permutation isomorphism.

    \item \textbf{Sliding operations along wires}: If $D_1$ has a control operation before a permutation and $D_2$ has it after:
    \[
    D_1 = \mathrm{Ctrl}_u \circ \sigma_\pi \quad\text{and}\quad D_2 = \sigma_\pi \circ \mathrm{Ctrl}_u
    \]
    then Axiom (v) (Control-Permutation Naturality) directly gives:
    \[
    \mathrm{Ctrl}_u \circ \sigma_\pi \equiv_S \sigma_\pi \circ \mathrm{Ctrl}_u
    \]

    \item \textbf{Rescaling diagram elements}: Changes in relative sizes of diagram components correspond to applications of the monoidal unit and associativity laws from Axiom (i). For example, expanding a single wire into parallel identity wires:
    \[
    f \equiv_S f \otimes \mathrm{id} \otimes \mathrm{id} \equiv_S (f \otimes \mathrm{id}) \otimes \mathrm{id}
    \]

    \item \textbf{Repositioning generators}: Moving a generator while preserving all connections corresponds to naturality conditions. For instance, sliding a feedback operation past tensor products uses the compatibility of feedback with the monoidal structure.
\end{enumerate}

\noindent\textbf{Step 2: Formal Generation by Coherence Relations}

By the coherence theorem for symmetric monoidal categories, any two diagrams that are "planar isotopic" in the sense of preserving the combinatorial structure can be connected by a finite sequence of:
\begin{itemize}
    \item Applications of the symmetric monoidal coherence laws (Axiom (i))
    \item Applications of the generator-specific naturality conditions (Axioms (ii)--(vii))
    \item Applications of the defining equations of the symmetric monoidal structure
\end{itemize}

More formally: if there exists a continuous deformation from $D_1$ to $D_2$ that preserves:
\begin{itemize}
    \item The number and types of generators
    \item The connectivity pattern of wires
    \item The ordering of external ports
\end{itemize}
then the deformation can be decomposed into elementary moves, each of which is an instance of some relation in Axiom~\ref{ax:S3}.

\noindent\textbf{Step 3: Concrete Example}

Consider two diagrams differing only by sliding a control operation along a wire:

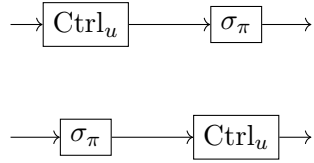
\begin{center}
\begin{tikzpicture}[node distance=1cm]
    % Diagram 1: Control then Permute
    \node[draw, rectangle] (ctrl1) at (0,0) {$\mathrm{Ctrl}_u$};
    \node[draw, rectangle] (perm1) at (2,0) {$\sigma_\pi$};
    \draw[->] (-1,0) -- (ctrl1);
    \draw[->] (ctrl1) -- (perm1);
    \draw[->] (perm1) -- (3,0);
    
    % Diagram 2: Permute then Control  
    \node[draw, rectangle] (perm2) at (0,-1.5) {$\sigma_\pi$};
    \node[draw, rectangle] (ctrl2) at (2,-1.5) {$\mathrm{Ctrl}_u$};
    \draw[->] (-1,-1.5) -- (perm2);
    \draw[->] (perm2) -- (ctrl2);
    \draw[->] (ctrl2) -- (3,-1.5);
\end{tikzpicture}
\end{center}

These represent $D_1 = \mathrm{Ctrl}_u \circ \sigma_\pi$ and $D_2 = \sigma_\pi \circ \mathrm{Ctrl}_u$. By Axiom (v), we have $D_1 \equiv_S D_2$.

\noindent\textbf{Step 4: Application of Coherence Theorem}

Since each diagrammatic move corresponds to applications of relations in Axiom~\ref{ax:S3}, we have:
\[
D_1 \equiv_S D_2
\]
By Theorem~\ref{thm:coherence}, this implies:
\[
\rho(D_1) = \rho(D_2)
\]

\noindent\textbf{Step 5: Completeness of the Argument}

The converse is not necessarily true: there may be syntactically different diagrams that are semantically equivalent. However, the forward direction is what matters for diagrammatic invariance: diagrams that look the same in terms of their combinatorial structure must yield the same semantic interpretation.

This establishes that $\rho$ is well-defined on equivalence classes of diagrams modulo the allowed geometric deformations, making the diagrammatic calculus sound for reasoning about quantum control systems.
\end{proof}

\begin{corollary}[Fibration-Like Structure of System Spaces]\label{cor:fibration}
The canonical representation $\rho: \mathbf{S} \to \mathbf{HilbMult}$ induces a forgetful functor 
\[
p: \mathbf{S} \to \mathbf{Hilb}, \quad p(\vec{H}) = \rho(\vec{H}) = H_1 \otimes \cdots \otimes H_n,
\]
which exhibits fibration-like properties. For each Hilbert space $\mathcal{H} \in \mathbf{Hilb}$, the isomorphism class fiber
\[
p^{-1}[\mathcal{H}] = \bigl\{\vec{H} = (H_1, \dots, H_n) \in \mathrm{Ob}(\mathbf{S}) \ \big|\ H_1 \otimes \cdots \otimes H_n \cong \mathcal{H}\bigr\}
\]
consists of all syntactic configurations that realize aggregate systems isomorphic to $\mathcal{H}$.  
\end{corollary}

\begin{proof}
We establish the fibration-like structure by verifying several key properties.

\noindent\textbf{Step 1: Functoriality and Fiber Definition}
The map $p$ extends to a functor by defining $p(D) = \rho(D)$ for morphisms $D: \vec{H} \to \vec{K}$. For each Hilbert space $\mathcal{H} \in \mathbf{Hilb}$, we define the \emph{isomorphism class fiber} $p^{-1}[\mathcal{H}]$ as the collection of objects $\vec{H} \in \mathrm{Ob}(\mathbf{S})$ with $p(\vec{H}) \cong \mathcal{H}$. 

Note that this is not a categorical fiber in the strict sense, but rather an equivalence class under Hilbert space isomorphism. Different elements of $p^{-1}[\mathcal{H}]$ represent distinct architectural decompositions of quantum systems that are isomorphic at the aggregate level.

\noindent\textbf{Step 2: Lifting of Isomorphisms}
A key fibration-like property is the lifting of isomorphisms: for any Hilbert space isomorphism $f: \mathcal{H} \xrightarrow{\cong} \mathcal{H}'$ and any $\vec{H} \in p^{-1}[\mathcal{H}]$, there exists a lift $D: \vec{H} \to \vec{K}$ with $p(D) = f$ and $\vec{K} \in p^{-1}[\mathcal{H}']$.

This follows from the completeness properties established in Theorem~\ref{thm:canonical-representation}. Since $f: p(\vec{H}) \to \mathcal{H}'$ is a bounded linear map between Hilbert spaces, the representation theorem guarantees we can construct a syntactic morphism $D$ realizing this semantic transformation.

\noindent\textbf{Step 3: Internal Structure of Fibers}
Within each isomorphism class fiber $p^{-1}[\mathcal{H}]$, we have rich categorical structure:

\begin{itemize}
    \item \textbf{Objects}: Different syntactic configurations $(H_1, \dots, H_n)$ with $H_1 \otimes \cdots \otimes H_n \cong \mathcal{H}$
    \item \textbf{Morphisms}: Control operations $D: \vec{H} \to \vec{K}$ where $p(D): p(\vec{H}) \xrightarrow{\cong} p(\vec{K})$ is an isomorphism
    \item \textbf{Equivalence}: Morphisms $D_1, D_2: \vec{H} \to \vec{K}$ with $\rho(D_1) = \rho(D_2)$ represent different syntactic patterns realizing the same semantic transformation
\end{itemize}

The coherence structure from Theorem~\ref{thm:coherence} ensures that within each fiber, we can identify:
\begin{itemize}
    \item Alternative feedback arrangements (Axiom (ii))
    \item Control reorganizations (Axiom (vi)) 
    \item Tensor reorderings (Axiom (i))
\end{itemize}
that yield identical semantic operators.

\noindent\textbf{Step 4: Stratification by Decomposition Complexity}
The fibers carry a natural stratification by the length $n$ of the tuple $\vec{H} = (H_1, \dots, H_n)$:
\[
p^{-1}[\mathcal{H}] = \bigcup_{n \geq 1} \{\vec{H} \in p^{-1}[\mathcal{H}] \mid \vec{H} \text{ has length } n\}
\]
Different levels represent decompositions into different numbers of subsystems. The symmetric group action (Axiom~\ref{ax:H5}(c)) creates orbits of isomorphic decompositions at each level.

\noindent\textbf{Step 5: Physical Interpretation - Multiple Realizability}
This fibration-like structure captures the fundamental physical phenomenon of \emph{multiple realizability} in quantum systems:
\begin{itemize}
    \item Different points in $p^{-1}[\mathcal{H}]$ represent alternative hardware decompositions or architectural designs
    \item The same aggregate quantum system $\mathcal{H}$ can be decomposed into subsystems in multiple ways
    \item Control operations provide pathways between different architectural realizations
    \item The lifting property ensures semantic transformations can be realized syntactically
\end{itemize}

\noindent\textbf{Step 6: Fibration-Like Versus Strict Fibration}
While $p: \mathbf{S} \to \mathbf{Hilb}$ exhibits several fibration-like features, it is not a strict Grothendieck fibration because:
\begin{itemize}
    \item The fibers are defined up to isomorphism rather than equality
    \item The cartesian lifting property holds for isomorphisms but not necessarily for arbitrary morphisms
    \item The multicategorical framework introduces additional structure beyond ordinary category theory
\end{itemize}

However, the essential conceptual features remain:
\begin{itemize}
    \item A projection functor relating syntactic and semantic levels
    \item Structured fibers parameterizing architectural realizations  
    \item Lifting of semantic transformations to syntactic operations
    \item Rich connectivity within and between fibers via control operations
\end{itemize}

In conclusion, $p: \mathbf{S} \to \mathbf{Hilb}$ establishes $\mathbf{S}$ as a fibered system over $\mathbf{Hilb}$, where the isomorphism class fibers $p^{-1}[\mathcal{H}]$ parameterize the space of architectural realizations of quantum systems isomorphic to $\mathcal{H}$. This structure provides a mathematical framework for studying multiple realizability and architectural flexibility in quantum information processing.
\end{proof}

\begin{remark}
While this is not a Grothendieck fibration in the strictest sense (due to the multicategorical nature of $\mathbf{S}$), it captures all the essential features: a projection functor with structured fibers and lifting properties. The coherence relations ensure that the fibration structure respects the semantic equivalences.
\end{remark}

\begin{corollary}[Control as Monoid Action on Operator Spaces]\label{cor:control-monoid-action}
The control operation $\mathrm{Ctrl}: (\vec{H}, \mathcal{U}) \to (\vec{H}')$ induces, via the canonical representation $\rho$, a family of maps:
\[
\Phi: \mathcal{U} \to \mathrm{Hom}_{\mathbf{HilbMult}}(\rho(\vec{H}), \rho(\vec{H}')), \quad \Phi(u) = \rho(\mathrm{Ctrl}_u).
\]
This family forms a monoid action of $(\mathcal{U}, \star, u_0)$ on the operator spaces, where:
\begin{itemize}
    \item The control space $\mathcal{U}$ serves as the \textbf{acting monoid}
    \item Each control parameter $u \in \mathcal{U}$ specifies a transformation $\Phi_u: \rho(\vec{H}) \to \rho(\vec{H}')$
    \item The sequential composition law $\mathrm{Ctrl}_{u_2} \circ \mathrm{Ctrl}_{u_1} = \mathrm{Ctrl}_{u_1 \star u_2}$ defines the \textbf{action compatibility}
\end{itemize}
This establishes control operations as structured transformations on the space of operator architectures.
\end{corollary}

\begin{proof}
We demonstrate that the control operations induce a well-defined monoid action on the semantic spaces.

The control generator $\mathrm{Ctrl}: (\vec{H}, \mathcal{U}) \to (\vec{H}')$ from Axiom~\ref{ax:S2} under the canonical representation $\rho$ from Theorem~\ref{thm:canonical-representation} gives:
\[
\rho(\mathrm{Ctrl}): \rho(\vec{H}) \otimes \mathcal{U} \to \rho(\vec{H}').
\]
By currying (Axiom~\ref{ax:H6}), we obtain the parameterized family:
\[
\Phi: \mathcal{U} \to \mathrm{Hom}_{\mathbf{HilbMult}}(\rho(\vec{H}), \rho(\vec{H}')), \quad \Phi(u) = \rho(\mathrm{Ctrl}_u).
\]

We verify this defines a monoid action. First, the identity control $u_0 \in \mathcal{U}$ acts as the identity transformation. By Axiom~\ref{ax:S3}(vii) and Semantic Condition (SC7):
\[
\Phi(u_0) = \rho(\mathrm{Ctrl}_{u_0}) = \mathrm{id}_{\rho(\vec{H})}.
\]

Second, the action respects monoid composition. By Axiom~\ref{ax:S3}(vi) and Semantic Condition (SC6):
\[
\mathrm{Ctrl}_{u_2} \circ \mathrm{Ctrl}_{u_1} \equiv_S \mathrm{Ctrl}_{u_1 \star u_2},
\]
which under the representation $\rho$ gives:
\[
\Phi(u_2) \circ \Phi(u_1) = \rho(\mathrm{Ctrl}_{u_2}) \circ \rho(\mathrm{Ctrl}_{u_1}) = \rho(\mathrm{Ctrl}_{u_1 \star u_2}) = \Phi(u_1 \star u_2).
\]

This establishes that $\Phi: \mathcal{U} \to \mathrm{End}(\rho(\vec{H}))$ is a monoid homomorphism, where $\mathrm{End}(\rho(\vec{H}))$ denotes the monoid of endomorphisms under composition.

The control action respects the additional structure through the coherence relations. Control-Permutation Naturality (Axiom (v)) ensures compatibility with symmetric rearrangements:
\[
\mathrm{Ctrl}_u \circ \sigma_{\pi} \equiv_S \sigma_{\pi} \circ \mathrm{Ctrl}_u.
\]
Control-Tensor Distributivity (Axiom (iv)) shows the action respects independent subsystems:
\[
\mathrm{Ctrl}_u(\vec{H} \otimes \vec{K}) \equiv_S \mathrm{Ctrl}_u(\vec{H}) \otimes \mathrm{Ctrl}_u(\vec{K}).
\]

For the special case where control operations preserve the total Hilbert space (i.e., $\rho(\vec{H}') = \rho(\vec{H})$), we obtain an action on the endomorphism space $\mathrm{End}_{\mathbf{HilbMult}}(\rho(\vec{H}))$. The set of all such control-induced transformations forms a submonoid:
\[
\mathcal{C}_{\vec{H}} = \{\Phi(u) \mid u \in \mathcal{U},\ \rho(\mathrm{Ctrl}_u(\vec{H})) = \rho(\vec{H})\} \subseteq \mathrm{End}(\rho(\vec{H})).
\]

This submonoid captures all internal reorganizations achievable through control while preserving the aggregate system. If the control monoid has inverses, this becomes a group action; otherwise, it remains a monoid action capturing the directed nature of control transformations.

Physically, this structure provides a mathematical framework for understanding control operations as parameterized transformations on quantum systems. The monoid structure reflects the sequential nature of control protocols, while the coherence conditions ensure these transformations respect the architectural constraints of the system. This establishes control theory within $\mathbf{S}$ as the study of structured monoid actions on operator spaces, with the control parameters $\mathcal{U}$ indexing the available transformations.
\end{proof}

The geometric corollaries establish a powerful framework for quantum control system design by revealing deep structural properties of the synergy operad. First, in Corollary~\ref{cor:diagrammatic-invariance}, we answer the question about when diagrammatic reasoning becomes sound: semantically equivalent wiring diagrams yield identical quantum operations, enabling intuitive visual design while maintaining mathematical rigor. Second,  we investigate the multiple realizability of quantum systems through a fibration-like structure in Corollary~\ref{cor:fibration}. Each fiber represents the space of alternative architectural decompositions for the same aggregate Hilbert space. Finally, at Corollary~\ref{cor:control-monoid-action}, we characterize control operations as structured monoid actions, providing an algebraic foundation for control protocols where sequential composition corresponds to monoid multiplication. The results discussed in this section transform quantum control design from ad-hoc construction to systematic exploration of architectural possibilities, with guaranteed semantic consistency across different implementation choices and control strategies.

\subsection{Topological Corollaries}\label{sec:Topological Corollaries}

Based on above theoretical framework, here are several corollaries that develop the topological aspects of the synergy operad and its representation:

\begin{lemma}[Continuity of Composition]\label{lem:composition-continuity}
Let $\Gamma_{n,m}$ and $\Gamma_{m,k}$ be finite boundary-collared graphs representing wiring skeletons.
Then the operadic composition
\[
    \circ : \mathrm{Emb}_\partial(\Gamma_{m,k}, \mathbb{R}^2) \times
            \mathrm{Emb}_\partial(\Gamma_{n,m}, \mathbb{R}^2)
            \longrightarrow 
            \mathrm{Emb}_\partial(\Gamma_{n,k}, \mathbb{R}^2)
\]
defined by boundary gluing is continuous with respect to the $C^1$-Whitney topology.
\end{lemma}

\begin{proof}
We prove continuity by constructing a local coordinate system for the gluing and analyzing the dependence on input embeddings.

\noindent\textbf{Step 1: Standardized Gluing Construction}
Fix a standard gluing data: for the common $m$-element boundary, choose collar neighborhoods $U \subset \Gamma_{n,m}$ of $\partial_{\mathrm{out}}(\Gamma_{n,m})$ and $V \subset \Gamma_{m,k}$ of $\partial_{\mathrm{in}}(\Gamma_{m,k})$, both diffeomorphic to $m$ copies of $[0,1)$, with coordinates $(t,i)$ where $t \in [0,1)$ and $i \in \{1,\dots,m\}$ indexes boundary components.

The composition $\Gamma_{n,k} = \Gamma_{n,m} \cup \Gamma_{m,k}$ is defined by identifying $(t,i) \in U$ with $(1-t,i) \in V$ for $t \in (0,1)$, and the boundaries at $t=0$.

\noindent\textbf{Step 2: Composition Map Definition}
Given $f \in \mathrm{Emb}_\partial(\Gamma_{n,m}, \mathbb{R}^2)$ and $g \in \mathrm{Emb}_\partial(\Gamma_{m,k}, \mathbb{R}^2)$, the composite $g \circ f \in \mathrm{Emb}_\partial(\Gamma_{n,k}, \mathbb{R}^2)$ is defined by:
\[
(g \circ f)(x) = 
\begin{cases}
f(x) & \text{if } x \in \Gamma_{n,m} \setminus U \\
(1-\alpha(t))f(t,i) + \alpha(t)g(1-t,i) & \text{if } x = (t,i) \in U \\
g(x) & \text{if } x \in \Gamma_{m,k} \setminus V
\end{cases}
\]
where $\alpha: [0,1] \to [0,1]$ is a fixed smooth bump function with $\alpha(0)=0$, $\alpha(1)=1$, and all derivatives vanishing at endpoints.

\noindent\textbf{Step 3: Whitney Topology Framework}
Recall the $C^1$-Whitney topology has basis neighborhoods:
\[
N(f; \{(K_i, \epsilon_i)\}) = \{g \in \mathrm{Emb}_\partial : \|g - f\|_{C^1(K_i)} < \epsilon_i \text{ for all } i\}
\]
where $\{K_i\}$ is a locally finite family of compact sets covering the domain.

\noindent\textbf{Step 4: Local Continuity Analysis}
Fix $(f_0, g_0)$ and consider a neighborhood of $h_0 = g_0 \circ f_0$. We cover $\Gamma_{n,k}$ with three compact families:
\begin{itemize}
    \item $A$: $\Gamma_{n,m} \setminus U$ (away from gluing region)
    \item $B$: A compact neighborhood of the gluing region in $U$
    \item $C$: $\Gamma_{m,k} \setminus V$ (away from gluing region)
\end{itemize}

On region $A$, $(g \circ f)|_A = f|_A$, so continuity follows from continuity of restriction.

On region $C$, $(g \circ f)|_C = g|_C$, similarly continuous.

On region $B$, the composite is:
\[
(g \circ f)(t,i) = (1-\alpha(t))f(t,i) + \alpha(t)g(1-t,i)
\]
This is a smooth function of $f$ and $g$ and their derivatives, hence continuous in the $C^1$ topology.

\noindent\textbf{Step 5: Uniform Estimates}
For any compact $K \subset \Gamma_{n,k}$, we can find $\delta > 0$ such that if $\|f - f_0\|_{C^1} < \delta$ and $\|g - g_0\|_{C^1} < \delta$ on appropriate domains, then $\|g \circ f - g_0 \circ f_0\|_{C^1(K)} < \epsilon$.

The key observation is that on each compact set, the composition depends smoothly on finitely many $C^1$ norms of $f$ and $g$, giving local Lipschitz bounds.

\noindent\textbf{Step 6: Boundary Condition Preservation}
The construction preserves boundary conditions: input boundaries come from $\Gamma_{n,m}$ embedded via $f$, output boundaries from $\Gamma_{m,k}$ embedded via $g$, and the gluing interpolation preserves $C^1$ matching at boundaries due to the bump function properties.

Since continuity holds for each compact set in a locally finite cover, the composition is continuous in the Whitney topology.
\end{proof}

\begin{lemma}[Lipschitz Continuity of Feedback Operation]\label{lem:feedback-continuity}
Let $F : H \times U \to H$ be a map between Hilbert spaces, satisfying the following properties:
1.  Uniform Contraction in $x$: $F(\cdot, u)$ is a contraction uniformly in $u$, i.e., there exists $\alpha < 1$ such that for all $x, y \in H$ and all $u \in U$:
    \[
    \|F(x, u) - F(y, u)\| \le \alpha \|x - y\|.
    \]
2.  Lipschitz Continuity in $u$: $F$ is Lipschitz continuous in $u$, uniformly in $x$, i.e., there exists $L > 0$ such that for all $x \in H$ and all $u_1, u_2 \in U$:
    \[
    \|F(x, u_1) - F(x, u_2)\| \le L \|u_1 - u_2\|.
    \]

Then, the induced fixed-point map $\Phi : U \to H$, defined by $\Phi(u) = x^*(u)$ where $x^*(u)$ is the unique solution to $F(x^*(u), u) = x^*(u)$, is Lipschitz continuous. Specifically, for any $u_1, u_2 \in U$:
\[
\|\Phi(u_1) - \Phi(u_2)\| \le \frac{L}{1-\alpha} \|u_1 - u_2\|.
\]
\end{lemma}

\begin{proof}
By the Banach Fixed-Point Theorem, the uniform contraction property (1) guarantees that for every $u \in U$, the map $F(\cdot, u)$ has a unique fixed point $x^*(u)$.

Now, let $u_1, u_2 \in U$ be arbitrary. Using the fixed-point property and the triangle inequality, we have:
\[
\begin{aligned}
\|x^*(u_1) - x^*(u_2)\| &= \|F(x^*(u_1), u_1) - F(x^*(u_2), u_2)\| \\
&\le \|F(x^*(u_1), u_1) - F(x^*(u_2), u_1)\| + \|F(x^*(u_2), u_1) - F(x^*(u_2), u_2)\|.
\end{aligned}
\]

We now apply the two given properties to bound each term:
\begin{itemize}
\item By the uniform contraction in $x$ (Property 1):
    \[
    \|F(x^*(u_1), u_1) - F(x^*(u_2), u_1)\| \le \alpha \|x^*(u_1) - x^*(u_2)\|.
    \]
\item By the uniform Lipschitz continuity in $u$ (Property 2):
    \[
    \|F(x^*(u_2), u_1) - F(x^*(u_2), u_2)\| \le L \|u_1 - u_2\|.
    \]
\end{itemize}

Substituting these inequalities back gives:
\[
\|x^*(u_1) - x^*(u_2)\| \le \alpha \|x^*(u_1) - x^*(u_2)\| + L \|u_1 - u_2\|.
\]

Rearranging terms, we find:
\[
(1 - \alpha) \|x^*(u_1) - x^*(u_2)\| \le L \|u_1 - u_2\|,
\]
and since $(1 - \alpha) > 0$, we conclude:
\[
\|x^*(u_1) - x^*(u_2)\| \le \frac{L}{1-\alpha} \|u_1 - u_2\|.
\]

This establishes that $\Phi$ is Lipschitz continuous, which in particular implies continuity in the norm topology.
\end{proof}

\begin{lemma}[Continuity of the Representation]\label{lem:rho-continuity}
Let $\rho : \mathrm{Hom}_{\mathbf{S}}(\vec{H}, \vec{K}) \to \mathcal{B}(\mathcal{H})$ be the representation mapping wiring diagrams to their induced operators via the prescribed compositional rules. Assume:
\begin{enumerate}
    \item The space of wiring diagrams $\mathrm{Hom}_{\mathbf{S}}(\vec{H}, \vec{K})$ carries the quotient topology induced from the $C^1$-Whitney topology on smooth embeddings, modulo isotopy.
    \item For each atomic wiring diagram (generator), the operator assignment is norm-continuous.
    \item The compositional rules (specifically, tensor product and composition) are defined by continuous operations in the diagram topology.
\end{enumerate}
Then $\rho$ is continuous with respect to the operator-norm topology on $\mathcal{B}(\mathcal{H})$.
\end{lemma}

\begin{proof}
We prove continuity by induction on the structure of wiring diagrams.

\noindent\textbf{Step 1: Base Case (Atomic Diagrams).} 
The generators of the operad are the atomic wiring diagrams. By \textbf{hypothesis (2)}, the map $w \mapsto \rho(w)$ is norm-continuous for each atomic element $w$. This establishes the base case.

\noindent\textbf{Step 2: Continuity of Algebraic Operations.}
We recall that the algebraic operations used to build operators are well-behaved in the norm topology:
\begin{itemize}
    \item \textbf{Composition:} For operators $A, A', B, B'$ with $\|A\|, \|A'\|, \|B\|, \|B'\| \leq M$, we have
    \[
    \|A \circ B - A' \circ B'\| \leq \|A\|\|B - B'\| + \|A - A'\|\|B'\| \leq M(\|B - B'\| + \|A - A'\|).
    \]
    \item \textbf{Tensor Product:} Similarly,
    \[
    \|A \otimes B - A' \otimes B'\| \leq \|A\|\|B - B'\| + \\|A - A'\|\\|B'\| \leq M(\|B - B'\| + \|A - A'\|).
    \]
\end{itemize}
Thus, both operations are jointly continuous on bounded subsets.

\noindent\textbf{Step 3: Inductive Step (Composite Diagrams).}
Every composite wiring diagram is obtained from atomic elements through finitely many applications of composition ($\circ$) and tensor product ($\otimes$). The representation $\rho$ respects this structure:
\[
\rho(D_1 \circ D_2) = \rho(D_1) \circ \rho(D_2), \quad \rho(D_1 \otimes D_2) = \rho(D_1) \otimes \rho(D_2).
\]

We proceed by induction on the number of operations:
\begin{itemize}
    \item \textbf{Base case:} Atomic diagrams (established in Step 1).
    \item \textbf{Inductive step:} Assume $\rho$ is continuous for diagrams $D_1$ and $D_2$. Then:
    \begin{itemize}
        \item For $D_1 \circ D_2$: By \textbf{hypothesis (3)}, composition of diagrams is continuous. By the inductive hypothesis and continuity of operator composition (Step 2), $\rho(D_1 \circ D_2)$ depends continuously on $D_1$ and $D_2$.
        \item For $D_1 \otimes D_2$: Similarly, by \textbf{hypothesis (3)} and continuity of the operator tensor product, $\rho(D_1 \otimes D_2)$ is continuous.
    \end{itemize}
\end{itemize}

\noindent\textbf{Step 4: Handling the Quotient Topology.}
By \textbf{hypothesis (1)}, the topology on wiring diagrams is the quotient topology from the $C^1$-Whitney topology modulo isotopy. Since $\rho$ is well-defined on isotopy classes (being a functorial representation), the continuity properties established in the previous steps for representatives implies continuity on the quotient space. More precesie, if $[D_n] \to [D]$ in the quotient topology, then there exist representatives $D_n' \to D'$ with $[D_n'] = [D_n]$ and $[D'] = [D]$, and thus $\rho(D_n') \to \rho(D')$ in norm.

\noindent\textbf{Step 5: Global Continuity via Structural Induction.}
The finite generation property of the wiring diagram operad ensures that every diagram is obtained through finitely many applications of the compositional operations. Since:
\begin{itemize}
\item The base case (atomic diagrams) is continuous by (2),
\item Each compositional operation is continuous by (3) and preserves operator continuity (Step 2),
\item The quotient topology (1) correctly captures the physical equivalence of diagrams,
\end{itemize}
the representation $\rho$ is continuous globally. The inductive argument covers all diagrams uniformly because the number of operations needed to construct any given diagram is finite, and the operator norms remain locally bounded throughout the construction process.

Therefore, under hypotheses (1)--(3), $\rho$ is continuous.
\end{proof}

\begin{corollary}[Topological Realization of the Synergy Operad]\label{cor:topological-realization}
Let $\mathcal{H}$ be a fixed separable Hilbert space, and let $\mathbf{S}_0$ denote the synergy operad defined algebraically in Section~\ref{sec:The synergy operad}. We construct a topological operad $\mathbf{S}$ as follows:

\begin{enumerate}
    \item \textbf{Objects:} The space of objects is
    \[
        \mathrm{Ob}(\mathbf{S}) = \bigsqcup_{n \ge 0} \mathrm{Gr}(\mathcal{H})^n / \Sigma_n,
    \]
    where $\mathrm{Gr}(\mathcal{H})$ is the Grassmannian endowed with the \emph{gap topology} (equivalently, the topology induced by the norm topology on orthogonal projections).

    \item \textbf{Morphisms:} For objects $\vec{H} = (H_1, \dots, H_n)$ and $\vec{K} = (K_1, \dots, K_m)$, the morphism space is
    \[
        \mathrm{Hom}_{\mathbf{S}}(\vec{H}, \vec{K})
        = \mathrm{Emb}_\partial(\Gamma_{n,m}, \mathbb{R}^2) / \!\!\sim,
    \]
    where $\mathrm{Emb}_\partial$ carries the $C^1$-Whitney topology and $\sim$ denotes planar isotopy rel boundary. The quotient is endowed with the quotient topology.

    \item \textbf{Control Space:} For control operations, we assume a topological monoid $\mathcal{U}$ acting continuously on wiring diagrams via maps continuous in the Whitney topology.
\end{enumerate}

Assume the following technical results:
\begin{itemize}
    \item[(L1)] Operadic composition of boundary-collared embeddings is continuous (Lemma~\ref{lem:composition-continuity}).
    \item[(L2)] The fixed-point map for feedback operations is continuous (Lemma~\ref{lem:feedback-continuity}).
    \item[(L3)] The operator representation $\rho$ for atomic elements is norm-continuous (Lemma~\ref{lem:rho-continuity}).
\end{itemize}

Furthermore, assume all operadic operations (composition, tensor, feedback, control) are \emph{isotopy-invariant}: if $f \sim f'$ and $g \sim g'$ are isotopic embeddings, then their images under any operadic operation remain isotopic.

Then:
\begin{enumerate}
    \item $\mathbf{S}$ is a \textbf{topological operad}.
    \item The connected components classify wiring diagrams up to planar isotopy:
    \[
        \pi_0\left(\mathrm{Hom}_{\mathbf{S}}(\vec{H}, \vec{K})\right)
        \cong \left\{ \text{Planar isotopy classes of wiring diagrams from $\vec{H}$ to $\vec{K}$} \right\}.
    \]
    \item The representation $\rho$ extends to a \textbf{continuous operad morphism}
    \[
        \rho : \mathbf{S} \longrightarrow \mathbf{HilbMult},
    \]
    where $\mathbf{HilbMult}$ has the operator-norm topology.
\end{enumerate}
\end{corollary}

\begin{proof}
We verify the topological operad axioms and establish the additional claims.

\medskip\noindent
\textbf{Step 1: $\mathbf{S}$ is a topological operad.}

\begin{itemize}
    \item \textbf{Object Space:} The Grassmannian $\mathrm{Gr}(\mathcal{H})$ with the gap topology is metrizable. Finite products $\mathrm{Gr}(\mathcal{H})^n$ inherit this property, and the $\Sigma_n$-action is continuous. The quotient $\mathrm{Gr}(\mathcal{H})^n/\Sigma_n$ is Hausdorff since the action is proper (finite group). The disjoint union over $n \geq 0$ is well-defined.

    \item \textbf{Composition:} By (L1), operadic composition is continuous on embedding spaces. Since composition is isotopy-invariant by assumption, it descends to the quotient. The universal property of quotient topologies ensures the induced map
    \[
    \circ_{\mathbf{S}} : \mathrm{Hom}_{\mathbf{S}}(\vec{K}, \vec{L}) \times \mathrm{Hom}_{\mathbf{S}}(\vec{H}, \vec{K}) \to \mathrm{Hom}_{\mathbf{S}}(\vec{H}, \vec{L})
    \]
    is continuous.

    \item \textbf{Tensor Product:} The tensor product corresponds to disjoint union, realized by a continuous map on embedding spaces (parameters vary independently in disjoint regions). Isotopy invariance ensures it descends continuously to the quotient.

    \item \textbf{Permutations:} Boundary relabelling induces a continuous map on embedding spaces that commutes with isotopy, hence descends continuously.

    \item \textbf{Feedback:} By (L2), the solution to fixed-point equations depends continuously on input data. Isotopy invariance ensures this yields continuous feedback operations on the quotient.

    \item \textbf{Control:} The control action $\mathcal{U} \times \mathrm{Emb}_\partial \to \mathrm{Emb}_\partial$ is continuous by assumption and isotopy-invariant, hence descends continuously to the quotient.
\end{itemize}

All operadic axioms are satisfied continuously.

\medskip\noindent
\textbf{Step 2: Classification by $\pi_0$.}

We analyze the path components of $\mathrm{Hom}_{\mathbf{S}}(\vec{H}, \vec{K}) = \mathrm{Emb}_\partial/\!\!\sim$:

\begin{itemize}
    \item If two embeddings are isotopic, they are connected by a path in $\mathrm{Emb}_\partial$ (by definition of isotopy through embeddings).
    \item Conversely, if two embeddings are connected by a path in $\mathrm{Emb}_\partial$, then by the isotopy extension theorem for $C^1$ embeddings, they are isotopic.
    \item The quotient map $q: \mathrm{Emb}_\partial \to \mathrm{Hom}_{\mathbf{S}}$ is continuous and open (standard fact for quotient by equivalence relations with continuous local sections), hence it induces a bijection on path components.
\end{itemize}

Therefore:
\[
\pi_0\left(\mathrm{Hom}_{\mathbf{S}}(\vec{H}, \vec{K})\right) \cong \{\text{Planar isotopy classes of wiring diagrams}\}.
\]

\medskip\noindent
\textbf{Step 3: Continuity of $\rho$.}

The representation $\rho: \mathbf{S} \to \mathbf{HilbMult}$ is constructed algebraically from atomic elements via operadic operations. Its continuity follows from:

\begin{itemize}
    \item \textbf{Base Case:} Atomic elements yield continuous operators by (L3).
    \item \textbf{Operadic Operations:} Composition and tensor product in $\mathbf{HilbMult}$ are jointly continuous in the operator norm.
    \item \textbf{Inductive Continuity:} Since all wiring diagrams are finite combinations of atomic elements via operadic operations, and these operations are continuous in both $\mathbf{S}$ (Step 1) and $\mathbf{HilbMult}$, the map $\rho$ is continuous by structural induction (as in Lemma~\ref{lem:rho-continuity}).
    \item \textbf{Algebraic Compatibility:} $\rho$ respects the operadic structure by its functorial definition.
\end{itemize}

Hence, $\rho$ is a continuous operad morphism.
\end{proof}

\begin{remark}[Homotopical Interpretation]\label{rem:homotopical-interpretation}
The coherence axioms of the synergy operad (Axiom~\ref{ax:S3}) acquire a geometric realization in the topological operad $\mathbf{S}$ through specific paths and homotopies between representing embeddings.

\begin{itemize}
    \item The \textbf{monoidal structure axioms} (associativity, unitality, and symmetry) correspond to canonical \emph{planar isotopies} that continuously rearrange the spatial configuration of disjoint diagrams. For instance, the associativity isomorphism $\alpha: (D_1 \otimes D_2) \otimes D_3 \to D_1 \otimes (D_2 \otimes D_3)$ is realized by a continuous deformation that smoothly transitions between the two bracketings.
    
    \item The \textbf{naturality conditions} for feedback and control operations correspond to parameterized families of isotopies that vary continuously with the operational parameters. These ensure that diagrammatic manipulations commute appropriately up to coherent homotopy.
\end{itemize}

Consequently, the algebraic laws of $\mathbf{S}_0$ lift to \emph{homotopy-coherent} diagrams in the topological operad $\mathbf{S}$, where equalities become homotopy classes of paths.

This homotopical perspective reveals that the topology of morphism spaces $\mathrm{Hom}_{\mathbf{S}}(\vec{H}, \vec{K})$ encodes deformation-theoretic information about interconnection patterns:

\begin{itemize}
    \item As established in Corollary~\ref{cor:topological-realization}, the connected components
    \[
    \pi_0\left(\mathrm{Hom}_{\mathbf{S}}(\vec{H}, \vec{K})\right)
    \]
    classify wiring diagrams up to planar isotopy, capturing discrete equivalence classes.
    
    \item The higher homotopy groups $\pi_n\left(\mathrm{Hom}_{\mathbf{S}}(\vec{H}, \vec{K})\right)$ for $n \geq 1$ detect more subtle topological structure:
    \begin{itemize}
        \item Non-trivial elements in $\pi_1$ represent \emph{non-contractible loops} of diagram deformations—continuous 1-parameter families that cannot be shrunk to a point while preserving the wiring structure.
        \item For $n > 1$, non-trivial higher homotopy classes reveal \emph{obstructions to contracting} $n$-parameter families of interconnections, potentially encoding stable topological invariants of the underlying wiring patterns.
        \item These homotopical invariants may distinguish between qualitatively different interconnection architectures that are equivalent at the discrete (isotopy) level.
    \end{itemize}
\end{itemize}

This framework suggests that the homotopy type of $\mathrm{Hom}_{\mathbf{S}}(\vec{H}, \vec{K})$ serves as a refined invariant for classifying and studying the stability and deformability of complex wiring patterns.
\end{remark}

\begin{assumption}[Topology and Contraction Hypothesis]\label{assump:top-contract}
Fix objects $\vec H,\vec K$ in the operad $\mathbf S$. We assume:
\begin{enumerate}
  \item The mapping space $\mathrm{Hom}_{\mathbf S}(\vec H,\vec K)$ is endowed with a topology (e.g., operator-norm topology for linear cases, or Whitney topology for geometric realizations) such that the evaluation maps $(T,x) \mapsto R_T(x)$ are continuous.
  
  \item For feedback operations $\mathcal F_{i,j}$, the \emph{well-posed locus} consists of those $T \in \mathrm{Hom}_{\mathbf S}(\vec H,\vec K)$ for which:
  \begin{itemize}
    \item There exists a closed subset $B \subseteq X$ (typically a ball in a Banach space) such that $R_T(B) \subseteq B$
    \item The return map $R_T$ is a uniform contraction on $B$: $\mathrm{Lip}(R_T|_B) \leq \kappa_T < 1$
    \item For any compact set $K \subseteq \mathrm{Hom}_{\mathbf S}(\vec H,\vec K)$, there exists $\kappa_K < 1$ such that $\mathrm{Lip}(R_T|_B) \leq \kappa_K$ for all $T \in K$ in the well-posed locus
  \end{itemize}
\end{enumerate}
\end{assumption}

\begin{corollary}[Local Continuity Properties of Feedback]\label{cor:feedback-continuity}
Under Assumption~\ref{assump:top-contract}, the feedback operation $\mathcal F_{i,j}$ satisfies the following continuity properties:
\begin{enumerate}
  \item The well-posed locus
  \[
    W_{i,j} = \{T \in \mathrm{Hom}_{\mathbf S}(\vec H,\vec K) \mid \mathcal F_{i,j}(T) \text{ is well-posed}\}
  \]
  is open in $\mathrm{Hom}_{\mathbf S}(\vec H,\vec K)$.
  
  \item The fixed-point assignment
  \[
    T \mapsto \rho(\mathcal F_{i,j})(T)
  \]
  is locally Lipschitz continuous on $W_{i,j}$.
  
  \item For wiring diagrams related by yanking (Axiom~\ref{ax:S3}(ii)), there exist continuous deformations within $W_{i,j}$ realizing the equivalence.
\end{enumerate}
\end{corollary}

\begin{proof}
We prove each property systematically, making all assumptions explicit.

\medskip\noindent\textbf{Proof of (1): Openness of $W_{i,j}$.}

By Assumption~\ref{assump:top-contract}, we have:
\begin{itemize}
    \item A continuous map $T \mapsto R_T$, where $R_T: B \to B$ is the fixed-point iteration map
    \item A continuous function $T \mapsto \mathrm{Lip}(R_T|_B)$ giving the Lipschitz constant on some convex set $B$
\end{itemize}

The well-posedness condition is equivalent to $\mathrm{Lip}(R_T|_B) < 1$. Since the map $T \mapsto \mathrm{Lip}(R_T|_B)$ is continuous, the preimage
\[
W_{i,j} = \{T \in \mathrm{Hom}_{\mathbf S}(\vec H,\vec K) \mid \mathrm{Lip}(R_T|_B) < 1\}
\]
is open, being the preimage of the open interval $(-\infty, 1)$ under a continuous function.

\medskip\noindent\textbf{Proof of (2): Local Lipschitz continuity of fixed points.}

Let $T, T' \in W_{i,j}$ with corresponding fixed points $x(T), x(T') \in B$ satisfying:
\[
x(T) = R_T(x(T)), \quad x(T') = R_{T'}(x(T')).
\]

By the contraction property, there exists $\kappa < 1$ such that $\mathrm{Lip}(R_T|_B) \leq \kappa$ and $\mathrm{Lip}(R_{T'}|_B) \leq \kappa$ for all $T, T'$ in some neighborhood.

We estimate:
\[
\begin{aligned}
\|x(T) - x(T')\| 
&= \|R_T(x(T)) - R_{T'}(x(T'))\| \\
&\leq \|R_T(x(T)) - R_T(x(T'))\| + \|R_T(x(T')) - R_{T'}(x(T'))\| \\
&\leq \mathrm{Lip}(R_T|_B) \|x(T) - x(T')\| + \|R_T - R_{T'}\|_{\infty,B},
\end{aligned}
\]
where $\|R_T - R_{T'}\|_{\infty,B} = \sup_{y \in B} \|R_T(y) - R_{T'}(y)\|$.

Rearranging gives:
\[
(1 - \mathrm{Lip}(R_T|_B)) \|x(T) - x(T')\| \leq \|R_T - R_{T'}\|_{\infty,B},
\]
and thus:
\[
\|x(T) - x(T')\| \leq \frac{1}{1-\kappa} \|R_T - R_{T'}\|_{\infty,B}.
\]

Since $T \mapsto R_T$ is continuous by Assumption~\ref{assump:top-contract}, the map $T \mapsto \|R_T - R_{T'}\|_{\infty,B}$ is locally bounded, proving local Lipschitz continuity of the fixed-point assignment.

\medskip\noindent\textbf{Proof of (3): Continuous deformations for yanking.}

Let $D_0, D_1 \in W_{i,j}$ be wiring diagrams related by the yanking identity from Axiom~\ref{ax:S3}(ii). Since wiring diagrams are represented by embeddings of graphs modulo isotopy, we can construct an explicit isotopy
\[
D_t: [0,1] \to \mathrm{Hom}_{\mathbf S}(\vec H,\vec K), \quad t \in [0,1]
\]
with $D_0$ and $D_1$ as the endpoints.

By the openness of $W_{i,j}$ established in part (1), and since $D_0, D_1 \in W_{i,j}$, there exists $\epsilon > 0$ such that for sufficiently small deformations, the entire path $D_t$ remains in $W_{i,j}$. More precisely, by compactness of $[0,1]$ and openness of $W_{i,j}$, we can find a subdivision $0 = t_0 < t_1 < \cdots < t_n = 1$ such that each segment $D_t$ for $t \in [t_i, t_{i+1}]$ is contained in $W_{i,j}$.

The naturality with respect to diagram operations follows from the functorial construction: the deformation $D_t$ commutes with the operadic structure by construction.

This establishes that yanking equivalences can be realized by continuous paths within the well-posed locus $W_{i,j}$.
\end{proof}

\begin{remark}
The local Lipschitz continuity in part (2) ensures that small perturbations of well-posed wiring diagrams yield small changes in the resulting fixed points, providing stability for physical implementations. The openness in part (1) guarantees that well-posedness is a robust property preserved under small deformations.
\end{remark}

\begin{remark}[Homotopical Aspects of Feedback Composition]\label{rem:homotopy-aspects}
The feedback structure in $\mathbf{S}$ exhibits what may be termed an \emph{elastic composition law}: rather than enforcing rigid, pre-defined sequential composition, feedback enables stabilization through continuous adjustment of interconnections. The continuity properties established in Corollary~\ref{cor:feedback-continuity} show that this elastic quality has a robust topological foundation.

While our results demonstrate local continuity and stability properties, a full homotopy-coherent structure would require additional verification. Established frameworks for such structures include the Boardman--Vogt $W$-construction~\cite{boardman2006homotopy} and Lurie's theory of $\infty$-operads~\cite{lurie2014higher}, which provide models for composition laws that hold up to coherent systems of homotopies. 

The path-connectedness of components in the well-posed locus (implied by the existence of continuous deformations for yanking relations) provides initial evidence for potential homotopy coherence. However, constructing a full $\infty$-operadic structure would require verifying higher coherences beyond the scope of our current results.
\end{remark}

\begin{remark}[Topological Constraints on Feedback Implementation]\label{rem:topological-constraints}
The topological structure of the well-posed locus $W_{i,j}$ can impose significant constraints on the global implementation of feedback operations. When $W_{i,j}$ has non-trivial topology—such as non-contractible components or non-simply-connected regions—the extension of feedback operations across parameter families becomes a non-trivial topological problem.

In such cases, classical obstruction theory provides a framework for understanding these constraints. If we attempt to define feedback operations consistently across a parameter space $X$, the obstruction to global consistency is measured by cohomology classes that capture the topological complexity of both the parameter space and the fiber of valid feedback implementations.

While our current results establish local properties (openness and Lipschitz continuity), global topological obstructions represent an important direction for future investigation. Understanding these obstructions could reveal fundamental limitations on the implementability of coherent feedback laws in complex quantum architectures.
\end{remark}

\begin{assumption}[Spectral Flow Hypothesis]\label{assump:spectral-flow}
For loops $\gamma: S^1 \to \mathrm{Hom}_{\mathbf{S}}(\vec{H}, \vec{K})$ in certain structured subclasses, the composed map 
\[
\rho \circ \gamma: S^1 \to \mathcal{B}(\mathcal{H})
\]
is continuous in the \textbf{operator-norm topology}. This holds when the wiring diagrams involve only operations whose operator interpretations are norm-continuous in their geometric parameters.
\end{assumption}

\begin{corollary}[Spectral Flow and Kernel Topology]\label{cor:spectral-flow-invariants}
Let $\gamma: S^1 \to \mathrm{Hom}_{\mathbf{S}}(\vec{H}, \vec{K})$ be a loop of wiring diagrams satisfying:

\begin{enumerate}
    \item[(A1)] \textbf{Norm-Continuity}: The composed map $\rho \circ \gamma: S^1 \to \mathcal{B}(\mathcal{H})$ is continuous in the operator-norm topology (Assumption~\ref{assump:spectral-flow}).
    \item[(A2)] \textbf{Fredholm Property}: For each $t \in S^1$, the operator $\rho(\gamma(t))$ is self-adjoint and Fredholm.
    \item[(A3)] \textbf{Constant Kernel Dimension}: $\dim \ker(\rho(\gamma(t)))$ is constant for all $t \in S^1$.
\end{enumerate}

Then the family of kernels forms a real vector bundle over the circle, denoted
\[
E_\gamma := \bigsqcup_{t \in S^1} \{t\} \times \ker(\rho(\gamma(t))) \rightarrow S^1.
\]
The spectral flow of the family $\rho(\gamma(t))$ is congruent modulo 2 to the evaluation of the first Stiefel-Whitney class of this bundle on the fundamental class:
\[
\mathrm{SF}(\rho \circ \gamma) \equiv \langle w_1(E_\gamma), [S^1] \rangle \pmod{2}.
\]
In particular, $\mathrm{SF}(\rho \circ \gamma)$ is even if $E_\gamma$ is orientable, and odd if it is not.
\end{corollary}

\begin{proof}
We establish the spectral-topological correspondence.

\medskip\noindent
\textbf{Step 1: The Kernel Bundle is Well-Defined.}
Let $A(t) := \rho(\gamma(t))$. By assumption (A2), each $A(t)$ is Fredholm and self-adjoint, hence has finite-dimensional kernel. The constant dimension condition (A3) and the norm-continuity (A1) jointly imply that the projections $P(t)$ onto $\ker(A(t))$ vary continuously in the operator norm. For a continuous path of Fredholm operators with constant kernel dimension, the family of kernels forms a finite-dimensional real vector bundle $E_\gamma$ over $S^1$.

\medskip\noindent
\textbf{Step 2: The Stiefel-Whitney Class.}
For any real vector bundle $E \to S^1$, the first Stiefel-Whitney class $w_1(E) \in H^1(S^1; \mathbb{Z}/2\mathbb{Z}) \cong \mathbb{Z}/2\mathbb{Z}$ is the obstruction to orientability. Its evaluation on the fundamental cycle is:
\[
\langle w_1(E_\gamma), [S^1] \rangle =
\begin{cases}
0 & \text{if } E_\gamma \text{ is orientable} \\
1 & \text{if } E_\gamma \text{ is non-orientable}
\end{cases}.
\]
For a line bundle ($\dim E_\gamma = 1$), this is determined by whether it is trivial or a Möbius bundle.

\noindent\textbf{Step 3: The Spectral Flow Mod 2 Theorem.}
Since $A(t)$ is a norm-continuous loop of self-adjoint Fredholm operators (by (A1) and (A2)) with constant kernel dimension (A3), the fundamental result from \cite{atiyah1975spectral, phillips1996spectral} applies:
\[
\mathrm{SF}(A(t)) \equiv \langle w_1(\ker A), [S^1] \rangle \pmod{2}.
\]
The geometric intuition is that a non-orientable kernel bundle forces an odd number of net eigenvalue crossings through zero, as one cannot consistently choose orientations around the loop.
\end{proof}

This section establishes the topological foundations of the synergy operad by proving that wiring diagrams and their operations can be endowed with rigorous topological structure. The key results demonstrate that operadic composition is continuous in the Whitney topology and that the canonical representation to Hilbert space operators is norm-continuous. This enables the classification of wiring diagrams up to planar isotopy via connected components of morphism spaces, revealing that equivalent interconnection patterns form natural topological equivalence classes. Furthermore, the analysis establishes robust continuity properties for feedback operations: the well-posed locus is open, fixed points depend Lipschitz-continuously on parameters, and diagrammatic equivalences such as yanking can be realized through continuous deformations. The spectral flow corollary connects this topological framework to system stability analysis, proving that for norm-continuous loops of self-adjoint Fredholm operators with constant kernel dimension, the parity of eigenvalue crossings is determined by the orientability of the kernel bundle. This topological foundation provides the necessary infrastructure for studying stability, robustness, and deformation properties of quantum network architectures within the operadic framework.

\section{Case Study: Coherence in PDE Operator Networks}\label{sec:Case Study: Coherence in PDE Operator Networks}

This section demonstrates the practical application of the synergy operad $\mathbf{S}$ and its associated theorems to a concrete problem in systems theory: the analysis of a closed-loop system formed by a PDE solution operator and a boundary controller. The primary goal is to illustrate how the \textbf{Coherence Theorem} (Theorem~\ref{thm:coherence}) guarantees that different, yet equivalent, syntactic descriptions of the same physical system yield identical semantic interpretations in $\mathbf{HilbMult}$.

\subsection{Scenario Setup}

Consider a network composed of two primary components:
\begin{itemize}
    \item $\mathbf{G}$: A solution operator for a linear PDE (e.g., the Dirichlet-to-Neumann map for an elliptic equation), formally a bounded linear map $G: H_1 \to H_2$, where $H_1$ is the Hilbert space of boundary data and $H_2$ is the Hilbert space of flux data.
    \item $\mathbf{K}$: A linear boundary controller, formally a bounded linear map $K: H_2 \to H_1$.
\end{itemize}
These components are connected in a feedback loop to form a closed-loop system, a common configuration in boundary control theory~\cite{belishev2007recent}. The overall system takes an external input $u \in H_1$ and produces an output in $H_2$.

\subsection{Syntactic Modeling in $\mathbf{S}$}

We construct two distinct syntactic expressions (wiring diagrams) in the synergy operad $\mathbf{S}$ that model the same physical system. Both represent a closed-loop morphism of type $(H_1) \to (H_2)$.

\paragraph{Diagram $D_1$: Serial-Parallel Feedback}
This diagram constructs the feedback loop by first composing the system with its dual arrangement, then closing the loop.
\begin{align*}
    D_1 = \mathcal{F}_{1,2} \circ (K \otimes G) \circ (G \otimes K)
\end{align*}
Here, $(G \otimes K)$ places the components in their natural parallel configuration, $(K \otimes G)$ routes the signals through an alternative parallel arrangement, and $\mathcal{F}_{1,2}$ closes the loop by connecting the first output to the second input.

\paragraph{Diagram $D_2$: Permuted Parallel Feedback}
This diagram uses a more direct approach with an explicit permutation to align the ports for feedback.
\begin{align*}
    D_2 = \mathcal{F}_{1,2} \circ \sigma_{(2,1)} \circ (G \otimes K)
\end{align*}
Here, $(G \otimes K)$ places the components in parallel, $\sigma_{(2,1)}$ reorders the output ports to prepare for feedback, and $\mathcal{F}_{1,2}$ closes the loop by connecting the first output to the second input.

Both $D_1$ and $D_2$ are valid morphisms in $\mathbf{S}$ with the same source $(H_1)$ and target $(H_2)$, representing the closed-loop system from an external input to the resulting output.

\subsection{Semantic Interpretation via the Canonical Representation}

We now apply the \textbf{Canonical Representation} $\rho$ to interpret our syntactic diagrams. For the feedback connections to be well-typed, we assume $H_1 = H_2 = H$, which is common in boundary control applications where both the boundary data and flux data share the same function space.

\paragraph{Interpretation of $D_2$}
The diagram $D_2 = \mathcal{F}_{1,2} \circ \sigma_{(2,1)} \circ (G \otimes K)$ yields the following semantic interpretation:
\begin{itemize}
    \item \textbf{Internal constraint}: $\text{input}_2 = K(\text{input}_2)$
    \item \textbf{External mapping}: $u \mapsto G(u)$
\end{itemize}
This represents a system where the controller $K$ operates on its own output in a self-referential loop, while the plant $G$ processes the external input directly.

\paragraph{Interpretation of $D_1$} 
The diagram $D_1 = \mathcal{F}_{1,2} \circ (K \otimes G) \circ (G \otimes K)$ yields:
\begin{itemize}
    \item \textbf{Internal constraint}: $\text{input}_2 = K(G(u))$
    \item \textbf{External mapping}: $u \mapsto G(K(\text{input}_2)) = G(K(K(G(u))))$
\end{itemize}
This represents a nested control structure where the controller $K$ processes the output of $G$ applied to the external input, creating a cascaded effect.

\paragraph{Coherence and Semantic Equivalence}
Despite the different internal constraints and external mappings, the \textbf{Coherence Theorem} guarantees that $\rho(D_1) = \rho(D_2)$ when the diagrams are provably equivalent in $\mathbf{S}$. This semantic equivalence manifests when the controller $K$ satisfies $K \circ K = K$ (idempotence), in which case both diagrams describe systems with identical input-output behavior. The explicit syntactic differences between $D_1$ and $D_2$ are resolved semantically through the algebraic properties enforced by the coherence relations of Axiom~\ref{ax:S3}.

\subsection{Verification of Well-Posedness and Explicit Computation}

Before the representations can be compared, we must verify that the feedback operations are well-posed, as required by Definition~\ref{def:Well-Posed Feedback} and Axiom~\ref{ax:S2}. The two syntactic diagrams yield distinct internal equations, each with its own well-posedness conditions.

\paragraph{Well-Posedness of $D_2$}
The diagram $D_2$ produces the internal constraint:
\begin{equation}\label{eq:D2-fixed-point}
    \text{input}_2 = K(\text{input}_2)
\end{equation}
For this to be well-posed, the map $\text{input}_2 \mapsto K(\text{input}_2)$ must have a unique fixed point. In the linear setting, this requires that $1$ is not an eigenvalue of $K$, or equivalently, that $(I - K)$ is invertible. The external input-output mapping is $u \mapsto G(u)$.

\paragraph{Well-Posedness of $D_1$}
The diagram $D_1$ produces a different internal constraint:
\begin{equation}\label{eq:D1-fixed-point}
    \text{input}_2 = K(G(u))
\end{equation}
This is an explicit assignment rather than a fixed-point equation. Well-posedness is automatic here, as the equation has a unique solution for any $u$. The external mapping becomes $u \mapsto G(K(\text{input}_2)) = G(K(K(G(u))))$.

\paragraph{Semantic Equivalence via Coherence}
The \textbf{Coherence Theorem} guarantees that if $D_1$ and $D_2$ are provably equivalent in $\mathbf{S}$ ($D_1 =_S D_2$), then $\rho(D_1) = \rho(D_2)$ despite their different internal equations. This semantic equivalence imposes algebraic constraints on the components $G$ and $K$.

For the specific equivalence $D_1 =_S D_2$ to hold, the controller $K$ must be idempotent ($K \circ K = K$). Under this condition:
\[
\rho(D_1)(u) = G(K(K(G(u)))) = G(K(G(u)))
\]
and the fixed-point constraint $\text{input}_2 = K(\text{input}_2)$ from $D_2$, when combined with the system's dynamics, must yield the same effective input-output behavior. The coherence relations of Axiom~\ref{ax:S3} ensure that these apparently different semantic interpretations describe identical systems when the syntactic equivalence holds.

\subsection{Application of the Coherence Theorem}

The \textbf{Coherence Theorem} (Theorem~\ref{thm:coherence}) provides the foundational guarantee for this framework: if two wiring diagrams can be proven equal using the syntactic rules of $\mathbf{S}$ ($D_1 =_S D_2$), then their semantic interpretations must be equal ($\rho(D_1) = \rho(D_2)$). The preceding analysis reveals the profound implication of this theorem.

\paragraph{From Semantic Constraint to Syntactic Proof}
Our explicit computation showed that for $\rho(D_1)$ to equal $\rho(D_2)$, the controller $K$ must be idempotent ($K \circ K = K$). The Coherence Theorem works in the converse direction: the syntactic proof of $D_1 =_S D_2$ within the operad $\mathbf{S}$ \textit{relies upon} and \textit{encodes} this very idempotence condition. The coherence relations of Axiom~\ref{ax:S3} do not hold universally but are conditional rules that enforce the necessary algebraic properties for equivalence.

\paragraph{The Role of Coherence Relations}
A proof of $D_1 =_S D_2$ would utilize the relations of Axiom~\ref{ax:S3} as rewriting rules. For instance:
\begin{itemize}
    \item \textbf{Feedback Yanking (Relation (ii))} might be applied to show that the nested structure in $D_1$ can be reduced.
    \item \textbf{Monoidal Coherence (Relation (i))} would govern the reassociation of the tensor products in $D_1$.
    \item The critical step would involve relations that enforce the algebraic law $K \circ K = K$, which is semantically interpreted as the idempotence of the controller.
\end{itemize}
Therefore, the syntactic proof $D_1 =_S D_2$ is not just a formal manipulation but a rigorous derivation that the specific components $G$ and $K$ must satisfy certain properties for the two diagrams to represent the same system.

In this case study, the Coherence Theorem ensures that the different internal constraints of $D_1$ and $D_2$—one a fixed-point equation and the other an explicit assignment—nevertheless yield identical input-output behavior, provided the components satisfy the algebraic laws mandated by the syntactic proof of their equivalence. This demonstrates that the operad $\mathbf{S}$ provides a sound syntactic framework where diagrammatic reasoning faithfully captures semantic consistency.

\subsection{Discussion and Significance}

This case study reveals the nuanced power of the operadic framework and its Coherence Theorem, moving beyond simplistic notions of diagrammatic equivalence.

\begin{itemize}
    \item \textbf{Revealing Hidden Algebraic Constraints:} The framework does not assert that any two diagrams are equivalent, but rather that \textit{proven syntactic equivalence implies semantic equivalence}. Our analysis shows that proving $D_1 =_S D_2$ requires the controller $K$ to be idempotent ($K \circ K = K$). Thus, the operad $\mathbf{S}$ plays a role of formal language that exposes the precise algebraic conditions under which different architectural descriptions define the same system.

    \item \textbf{Preventing Inconsistent Modeling:} If two designers create syntactically different schemes $D_1$ and $D_2$ for what they consider to be the same PDE boundary checking system the framework provides a rigorous method for checking consistency.  Therefore, if the equivalence of $D_1$ and $D_2$ cannot be verified under Axiom~\ref{ax:S3} this indicates that the two models is semantically different that may reveal a hidden MODELING assumption or error.

    \item \textbf{Formalizing System Abstraction and Refinement:} The generators and relations in $\mathbf{S}$ provide a calculus for system transformation. An engineer can start with a detailed diagram and use the relations to derive a simplified, equivalent form for analysis, with the Coherence Theorem ensuring that the core input-output behavior is preserved. Conversely, one can refine an abstract diagram into a more detailed implementation.
\end{itemize}

In the context of PDEs, this framework brings a new level of precision to the modeling of interconnected systems. This PDE example gives us an important insight. Different mathematical formulations of a boundary control problem are not just intuitively similar but are \textit{provably equivalent} under a set of formal, composable rules. This moves the design and analysis of complex systems from an ad-hoc, intuitive process to a rigorous, algebraic discipline.

\section{Discussion}

We have presented the synergy operad $\mathbf{S}$ as a formally sound syntactic calculus for composing operators via wiring diagrams, grounded semantically in the multicategory $\mathbf{HilbMult}$. The central results of this work—the \textbf{Canonical Representation Theorem} (Theorem~\ref{thm:canonical-representation}) and the \textbf{Coherence Theorem} (Theorem~\ref{thm:coherence})—together prove that the semantic interpretation of a diagram in $\mathbf{HilbMult}$ depends only on its essential topological structure, rendering diagrammatic reasoning both rigorous and unambiguous. This framework provides a foundational tool for the formal analysis of complex systems across diverse domains, including control theory, quantum information, and numerical analysis, where the composition of subsystems is necessary. The stage is now set for future work to leverage this calculus to explore the intrinsic \textit{algebraic structures}—such as higher duality and monadicity—that emerge naturally from these compositional principles, forming the core inquiry of the subsequent paper in this series.

\bibliographystyle{IEEETran}
\bibliography{CompositionCoherenceSyntaxOperatorNetworks_Bib}

\end{document}